\documentclass[reqno]{amsart} 
\usepackage{amssymb,amsmath,amsfonts}

\usepackage[hidelinks]{hyperref}

\allowdisplaybreaks
 
\newcommand{\norm}[1]{\left\Vert#1\right\Vert}
\newcommand{\abs}[1]{\left\vert#1\right\vert}

\newcommand{\set}[1]{\left\{#1\right\}}

\newtheorem{theorem}{Theorem}[section]
\newtheorem{corollary}[theorem]{Corollary}
\newtheorem{lemma}[theorem]{Lemma}

\newtheorem{remark}[theorem]{Remark}
\newtheorem{example}[theorem]{Example}

\numberwithin{equation}{section}
\numberwithin{table}{section}
\numberwithin{figure}{section}
\begin{document}
	
\title[Spectral methods for fractional reaction-diffusion equations]{Regularity and  spectral methods for  two-sided    fractional  diffusion equations  with a low-order term}

\author{Zhaopeng Hao}
	\address{Department of Mathematics, Southeast University, Nanjing 210096, P.R. China. } 
	\email{haochpeng@126.com} 
\author{Guang Lin}
			\address{Department of Mathematics and School of Mechanical Engineering, Purdue University,  West Lafayette, IN 47907 USA.}
 		   \email {guanglin@purdue.edu}
  \author{Zhongqiang Zhang}
\address{Department of Mathematics, Worcester Polytechnic Institute, Worcester, Ma 01609 USA.   Corresponding Author}
 \email {zzhang7@wpi.edu
 }

\date{\today}
\maketitle
\begin{abstract}
 We study  regularity and numerical  methods for  two-sided  fractional  diffusion equations with a lower-order term.  
 We show that the regularity of the solution  
 in weighted Sobolev spaces 
 can be  greatly improved  compared to  that in standard Sobolev spaces. 
 With this regularity, we improve  higher-order
 convergence   of a  spectral Galerkin method. We   present  a spectral  Petrov-Galerkin  method  and provide an optimal error estimate for the  Petrov-Galerkin  method. 
 Numerical results are presented to verify our  theoretical convergence orders. 	
	\end{abstract}
	


 \section{Introduction}
 Anomalous diffusion has been widely used to  investigate  transport dynamics in complex systems, such as underground environmental problem \cite{HatanoH98}, fluid flow in porous materials \cite{BensonSMW01},  anomalous transport in biology \cite{HoflingF13}, etc.  Many mathematical models are developed to study anomalous diffusion. Some of these models are based on a linear equation for diffusion on fractals \cite{O'Shaughnessy85},  a linear differential Fisher's information theory \cite{Ubriaco09} and Levy description of anomalous diffusion in dynamical systems\cite{KlafterZS95}. In particular,  fractional differential equations (FDEs) can serve as an accurate model of the anomalous diffusion, e.g.  super-diffusion process in \cite{MetzlerK00}.
 
 Explicit solutions for FDEs are  mostly 
 not available and thus it is essential to develop efficient numerical methods. Extensive numerical  methods have been investigated in recent decades
 e.g.  finite difference method \cite{HaoSC15,LiuAT04,MeerschaertT04,TianZD15,WangB12}, finite element method \cite{ErvinHR16,ErvinR06,JinLPR15,WangY13}, spectral method \cite{ChenSW16,ErvinHR16,LiX10,MaoCS16,HuangJWZ16,ZayernouriCZK14,ZayernouriK14}, discontinuous Galerkin method  \cite{XuH14},  finite volume method \cite{SimmonsYM17} etc.

 Despite of rich numerical methods for FDEs,  regularity of   solutions to FDEs is  not thoroughly investigated, especially the regularity    well suited 
 for error analysis.  In literature, it is assumed that 
   solutions are sufficiently smooth. However,   it has been
   pointed out in \cite{ChenSW16,JinLPR15,KoptevaS16,WangY13} that the  regularity of solutions to FDEs  can be very low.

 In this paper, we consider the following two-sided  fractional  diffusion equation with a reaction term
 \begin{equation}\label{eq:two-sided}
 \mathcal{L}_{\theta}^{\alpha}u+\mu u= f(x), \quad x\in I=(a,b),
 \end{equation}
 with homogeneous  boundary conditions
 \begin{equation}\label{eq:fode-riesz-non-symm-bc}
 u(a)=u(b)=0,
 \end{equation}
 and a given function  $f(x),$ 
 $\mathcal{L}_{\theta}^{\alpha}:=-[\theta\,  _{a}D_x^{\alpha}+(1-\theta)\,  _xD_b^{\alpha}] $ 
 with $\theta \in [0,1]$ and $\alpha \in (1,2) .$ Here   $_{a}D_x^{\alpha} $ and $ _xD_b^{\alpha}$ are left- and right-sided Riemann-Liouville operators,   defined as follows (see e.g.  \cite{Podlubny99,SamkoKM93})
  \begin{equation}\label{def-1}
 _{a}D_x^{\alpha} u(x) =\frac{1}{\Gamma(2-\alpha)}、\frac{d^2}{dx^2}\int_{a}^x \frac{u(\xi) }{(x-\xi)^{\alpha-1}} d\xi ,\quad x>a, 
 \end{equation}
 and
 \begin{equation}\label{def-1R}
 _xD_b^{\alpha} u(x) =\frac{1}{\Gamma(2-\alpha)}\frac{d^2}{dx^2}\int_x^b \frac{u(\xi) }{(\xi-x)^{\alpha-1}} d\xi ,\quad x<b.
 \end{equation}

 Ervin and Roop \cite{ErvinR06}  established uniqueness and  existence of  the  solution to \eqref{eq:two-sided}. 
 Ervin et al.  \cite{ErvinHR16} discussed the regularity of the solution  in standard Sobolev space for $\mu =0$ in \eqref{eq:two-sided}.  

 This work devotes to the study of regularity and  spectral methods for    \eqref{eq:two-sided}.  From the  weakly singular kernel  of fractional derivatives, solutions of fractional differential equations naturally  inherit   weak singularity.
 According to \cite{ErvinHR16}, the solution to \eqref{eq:two-sided} when $\mu=0$ can be written as the product
 $(1-x)^\gamma(1+x)^\beta \tilde{u}$, $\gamma$ and $\beta$ are 
 some constants depending on $\alpha$ and $\theta$ in \eqref{eq:two-sided}.  
 Based on the relation  in  Lemma \ref{lem:frac-non-symm-jacobi}, we expand the function $\tilde{u}$ using Jacobi polynomials and apply Fourier-type analysis of $\tilde{u}$.  For $\tilde{u}$, we show that the solution has a limited regularity, more precisely  $2\alpha+1$ in non-uniformly weighted Sobolev space even for smooth   $f$.  This work is motivated by   \cite{Zhang17} which  gave the regularity for the fractional diffusion equation with   fractional Laplacian by Fourier analysis and bootstrapping technique.
 We also notice that in    \cite{MaoKar17}, a detailed error estimate for  a Petrov-Galerkin method is given when $\mu=0$. 
 All these works can be considered as   special cases in the current work.

 We present  a spectral Petrov-Galerkin  method which is discussed in \cite{ErvinHR16,MaoKar17} where $\mu=0$ and also a spectral Galerkin  method.   Based on the obtained regularity, we prove error estimates of these methods. 
 For the spectral Galerkin method, an optimal error estimate is  obtained in  the case  $\theta=0.5$, similar to that in \cite{Zhang17}.  
 We  get optimal error estimate for the spectral Petrov-Galerkin method when $\mu$ is small; see Theorem  \ref{thm:spectral-error-estimate-2}.

 The rest of this paper is arranged as follows. In Section 2,
 we introduce some basic notations and  recall some properties
 of Jacobi polynomials and non-uniformly weighted Sobolev 
 spaces. Some lengthy but important auxiliary materials are
 presented in Appendix.   In Section 3, we present the
 regularity of the two-sided fractional diffusion equations
 using   Fourier type analysis   and a bootstrapping
 technique. We consider a spectral Galerkin method  and carry
 out its error analysis in Section 4.  In Section 5, we present a Petrov-Galerkin method  and provide error estimates as well. Several numerical results are showed to verify   the theoretical convergence order in Section 6.  Finally, we make some concluding remarks.

 \section{Preliminary}
 
 We consider the interval $I=(-1,1)$ for  simplicity.  Denote by $L^2_{\omega^{\gamma,\beta}} (I)$ the space with 
 the inner product and norm defined by
 $$(u,v)_{\omega^{\gamma,\beta},I}=\int_{I}{uv \omega^{\gamma,\beta}}dx,\quad \|u\|_{\omega,I}=(u,v)_{\omega,I}^{\frac{1}{2}} ,$$
 where $\omega^{\gamma,\beta}=(1-x)^\gamma(1+x)^\beta$, $\gamma,\beta>-1$. When $\gamma=\beta=0$, we will drop $\omega$ from the above notations. We  also drop  the domain $I$ from the notation for simplicity   without  incurring confusion.

 The Jacobi polynomials $P_n^{\gamma,\beta} (x)$ are mutually orthogonal:
 for $\gamma,\ \beta > -1$,
 \begin{equation}\label{eq:orthogonality-of-jacobi-poly}
 \int_{-1}^1 (1-x)^{\gamma} (1+x)^{\beta} P_m^{\gamma, \beta} (x)P_n^{\gamma,\beta} (x) \, dx =h_n^{\gamma,\beta}  \delta_{nm}.
 \end{equation}
 Here  $\delta_{nm}$ is equal to $1$ if $n=m$ and  zero otherwise,  and
 \begin{equation}\label{eq:orthogonality-of-jacobi-poly-normalized-constant}
 h_n^{\gamma,\beta}=\norm{P_n^{\gamma,\beta}}_{\omega^{\gamma,\beta}}^2 =\frac{2^{\gamma + \beta+1}}{2n+\gamma + \beta+1} \times \frac{\Gamma(n+\gamma+1)\Gamma(n+\beta+1)}{\Gamma(n+\gamma + \beta+1)n!}.
 \end{equation}

 To   incorporate singularities at the endpoints,   we  introduce the following non-uniformly   weighted Sobolev space, see e.g.  \cite{GuoW04},
 \begin{equation}
 B^{m}_{\omega^{\gamma,\beta}} :=\set{u \,|\, \partial_x^ku \in L^2_{\omega^{\gamma+k,\beta+k}},\, k=0,1,\ldots,m},~m\,\mbox{is a  nonnegative integer}
 \end{equation}
 which is equipped with the following norm
 \begin{equation}
 \norm{u}_{B^{m}_{\omega^{\gamma,\beta}}} =\big( \sum_{k=0}^m \abs{u}_{B^{k}_{\omega^{\gamma,\beta}}}^2\big )^{1/2}, \quad   \abs{u}_{B^{k}_{\omega^{\gamma,\beta}}}=\norm{\partial_x^k u}_{\omega^{\gamma+k,\beta+k}}.
 \end{equation}
 When $m$ is not an integer,   the space is defined  by interpolation, see e.g. \cite{GuoW04}.

 
 The following 
 \textit{pseudo-eigenfunctions} for the fractional diffusion operator in \cite{ErvinHR16} are essential to analyze the regularity.   
 \begin{lemma}[\cite{ErvinHR16,MaoKar17}]\label{lem:frac-non-symm-jacobi}
 	For the $n$-th order Jacobi polynomial $P_n^{\sigma,\,\sigma^{\ast}}(x)$, it holds that
 	\begin{eqnarray}
 	\mathcal{L}_{\theta}^{\alpha}[\omega^{\sigma,\sigma^{\ast}}(x) P_n^{\sigma,\sigma^{\ast}}(x)] &=& \lambda_{\theta,n}^{\alpha}P_n^{\sigma^{\ast},\,\sigma}(x),\\
 	\mathcal{L}_{\theta}^{\alpha}[\omega^{\sigma-1,\sigma^{\ast}-1}(x)]&=&0,
 	\end{eqnarray}
 	where  $$\lambda_{\theta,n}^{\alpha}= -\frac{\sin(\pi \alpha)}{\sin(\pi \sigma)+\sin(\pi\sigma^{\ast}) } \times \frac{\Gamma(\alpha+n+1)}{n!},$$ 
    and $\sigma^{\ast}=\alpha-\sigma$  and  $\sigma$ is determined by the following equation:
 	\begin{equation}\label{eq:fracal-sigma}
 	\theta=\frac{\sin(\pi \sigma^{\ast})}{\sin(\pi\sigma^{\ast})+\sin(\pi \sigma)}.
 	\end{equation}
 \end{lemma}
 
 \begin{remark}
 	To ensure that \eqref{eq:fracal-sigma} is uniquely solvable,
 	we restrict   $\sigma$ and $\sigma^{\ast}$  into  the interval  $(0,1].$
 	In particular,    $\sigma=1$ and $\sigma^{\ast}=\alpha-1$  for $\theta=1;$
 	$\sigma=\alpha-1$ and $\sigma^{\ast}=1$ for $\theta=0,$   and  $\lambda_{0,n}^{\alpha}=\lambda_{1,n}^{\alpha}= \frac{\Gamma(\alpha+n+1)}{n!}.$ 
 \end{remark}

 Throughout  the paper, $C$ or $c$ denote  generic constants
 and are independent of the truncation parameter $N$.   
  \section{Regularity }
 
 We present the regularity of the two-sided FDE
 \eqref{eq:two-sided}.
 \begin{theorem}[Regularity in weighted Sobolev spaces, I]\label{thm:fode-regularity-flap-nonsymm}  
 	Assume that $f\in B^{r}_{\omega^{\sigma^{\ast},\sigma}}$ and  $u\in L^\infty$.
 	If $\mu=0$,   then
 	$\omega^{-\sigma,-\sigma^{\ast}}u \in   B^{\alpha+r}_{ \omega^{\sigma,\sigma^{\ast}} }.$
 	If $\mu\neq0$, then
 	$\omega^{-\sigma,-\sigma^{\ast}}u \in   B^{ \alpha\wedge r+\alpha }_{ \omega^{\sigma,\sigma^{\ast}} }$. Here $\alpha\wedge r=\min\{\alpha, r\}.$
 \end{theorem}

 \begin{proof}
 	For $\mu =0$, we write  $ u =  \omega^{\sigma,\sigma^{\ast}}\sum_{n=0}^\infty   u_n      P_n^{\sigma,\sigma^{\ast}} (x)$. Then with Lemma \ref{lem:frac-non-symm-jacobi} we have
 	$u_n=(\lambda^{\alpha}_{\theta,n})^{-1} f_n $ from the equation $\mathcal{L}_{\theta}^{\alpha}u =f$ and $f = \sum_{n=0}^\infty   f_n      P_n^{\sigma^{\ast},\sigma} (x)$. Then  
 	\[u=   \omega^{\sigma,\sigma^{\ast}}\sum_{n=0}^\infty   (\lambda_{\theta,n}^{\alpha})^{-1}f_n      P_n^{\sigma,\sigma^{\ast}} (x).\]
 	By   \eqref{eq:jacobi-derivatives} and \eqref{estimate-gamma},  we have
 	\begin{eqnarray}\label{eq:weighted-norm-u}
 	&& \norm{ \partial_x^k  [\omega^{-\sigma,-\sigma^{\ast}}u ]}_{\omega^{\sigma+k,\sigma^{\ast}+k} }^2
 	=\norm{ \partial_x^k \sum_{n=1}^\infty (\lambda^{\alpha}_{\theta,n})^{-1}f_n P_n^{\sigma,\sigma^{\ast} }  }_{\omega^{\sigma+k,\sigma^{\ast}+k}  }^2 \notag \\
 	&&= \sum_{n=k}^\infty f_n^2(\lambda^{\alpha}_{\theta,n})^{-2}\big(d_{n,k}^{\sigma,\sigma^{\ast}}\big)^2  h_{n-k}^{\sigma+k,\sigma^{\ast}+k} \approx\sum_{n=k}^\infty f_n^2 n^{-2\alpha} \big(d_{n,k}^{\sigma,\sigma^{\ast}}\big)^2   h_{n-k}^{\sigma+k,\sigma^{\ast}+k},
 	\end{eqnarray}
 	where $d_{n,k}^{\gamma,\beta }=\frac{\Gamma(n+k+\gamma+\beta +1)}{2^k \Gamma(n+\gamma+\beta+1)}$.  Recall that we have
 	\begin{eqnarray}\label{eq:weighted-norm-f}
 	\norm{ \partial_x^k f}_{\omega^{\sigma^{\ast}+k,\sigma+k} }^2
 	&=& \norm{ \partial_x^k \sum_{n=1}^\infty f_n P_n^{\sigma^{\ast},\sigma }  }_{\omega^{\sigma^{\ast}+k,\sigma+k}  }^2=\sum_{n=k}^\infty f_n^2  \big(d_{n,k}^{\sigma^{\ast},\sigma}\big)^2  h_{n-k}^{\sigma^{\ast}+k,\sigma+k},\,\,\,
 	\end{eqnarray}
 		where $d_{n,k}^{\sigma^*,\sigma}=d_{n,k}^{\sigma,\sigma^*}$ and $h_{n-k}^{\sigma^*+k,\sigma+k}=h_{n-k}^{\sigma+k,\sigma^*+k}$  by \eqref{eq:orthogonality-of-jacobi-poly-normalized-constant}. 
 	By \eqref{eq:weighted-norm-u} and \eqref{eq:weighted-norm-f}, we  reach the conclusion from the
 	space interpolation theory.

 	Now consider $\mu\neq 0$. We use a bootstrapping technique to obtain higher regularity.  First,  	we can obtain that $u= \omega^{\sigma,\sigma^{\ast}}\tilde{u}$ and  $\tilde{u}\in B_{\omega^{\sigma,\sigma^{\ast}}}^{\alpha}$. In fact, we have from   $u\in L^\infty $ that   $\tilde{f}= f- \mu u \in  B_{\omega^{\sigma^{\ast},\sigma}}^{r\wedge 0} .$ 
 	From  the equation  $\mathcal{L}_{\theta}^{\alpha}  u = \tilde{f}$  and using  the conclusion when $\mu=0$, we have
 	$\omega^{-\sigma,-\sigma^{\ast}}u\in B_{\omega^{\sigma,\sigma^{\ast}}}^{r\wedge 0+\alpha }$.
 	Second, by a direct calculation using the definition of weighted space $B_{\omega^{\sigma^{\ast},\sigma}}^{r\wedge0+\alpha}$,  we have $u\in B_{\omega^{\sigma^{\ast},\sigma}}^{r\wedge0+\alpha}$.
 	Finally, by the fact that $\tilde{f}= f- \mu u  \in  B^{\alpha\wedge r }_{\omega^{\sigma^{\ast},\sigma}} $ and the conclusion for $\mu=0$,
 	we have
 	$\omega^{-\sigma,-\sigma^{\ast}} u\in B_{\omega^{\sigma,{\sigma^{\ast}}}}^{2\alpha \wedge (\alpha+r)}$.  	  
 \end{proof}

 \begin{theorem}[Regularity in weighted Sobolev spaces,   II]\label{thm:fode-regularity-two-side}  Assume that  
 	$f\in B^{r}_{\omega^{\sigma^{\ast}-1,\sigma-1}}$, with $r\geq 0$. If $\mu=0$,   then
 	$\omega^{-\sigma,-\sigma^{\ast}}u \in   B^{\alpha+r}_{\omega^{ \sigma-1,\sigma^{\ast}-1 }}.$ 
 	If $\mu\neq0$,
 	$\omega^{-\sigma,-\sigma^{\ast}}u \in   B^{ (\alpha+1)\wedge r+ \alpha }_{\omega^{ \sigma,\sigma^{\ast} }}$ and $\mathcal{L}_{\theta}^{\alpha}u\in B^{r  \wedge(\alpha+1)}_{\omega^{\sigma^{\ast},\sigma}}.$ 
 \end{theorem}
 \begin{proof}
 	Consider first $\mu=0$. 
 	From Corollary \ref{cor:different-index}, we have the following relations: for $n\geq 0$
 	\begin{eqnarray}
 	&&P_n^{\sigma-1,\sigma^{\ast}-1}=A_{n}^{\sigma-1,\sigma^{\ast}-1}P_{n-2}^{\sigma,\sigma^{\ast}}+B_{n}^{\sigma-1,\sigma^{\ast}-1}P_{n-1}^{\sigma,\sigma^{\ast}}+C_{n}^{\sigma-1,\sigma^{\ast}-1}P_{n}^{\sigma,\sigma^{\ast}},\label{eq:spectral-different-index}\\
 	&&P_n^{\sigma^{\ast}-1,\sigma-1}=A_{n}^{\sigma-1,\sigma^{\ast}-1}P_{n-2}^{\sigma^{\ast},\sigma}-B_{n}^{\sigma-1,\sigma^{\ast}-1}P_{n-1}^{\sigma^{\ast},\sigma}+C_{n}^{\sigma-1,\sigma^{\ast}-1}P_{n}^{\sigma^{\ast},\sigma},\label{eq:spectral-different-index-1}
 	\end{eqnarray}
 	where $A_{n}^{\sigma-1,\sigma^{\ast}-1},$
    $B_{n}^{\sigma-1,\sigma^{\ast}-1}$ and
    $C_{n}^{\sigma-1,\sigma^{\ast}-1}$ are defined in
    Corollary \ref{cor:different-index} and
    $P_{-2}^{\gamma,\beta}\equiv
    P_{-1}^{\gamma,\beta}\equiv0$.  Throughout the proof,  to
    simplify the notations but  without incurring confusion,
    we drop the superscript $\sigma-1,\sigma^{\ast}-1$ for
    $A_n$, $B_n$ and $C_n$ and abbreviate $\lambda_{\theta,n}^{\alpha}$ as $\lambda_n$. 
 From \eqref{eq:spectral-different-index}, we have  
 	\begin{eqnarray}\label{eq:spectral-a1}
 	\sum_{n=0}^{\infty}u_nP_n^{\sigma-1,\sigma^{\ast}-1}&=&\sum_{n=0}^{\infty}u_n(A_nP_{n-2}^{\sigma,\sigma^{\ast}}+B_nP_{n-1}^{\sigma,\sigma^{\ast}}+C_nP_{n}^{\sigma,\sigma^{\ast}})\nonumber \\
 	&=& \sum_{n=0}^{\infty}(u_{n+2}A_{n+2}+u_{n+1}B_{n+1}+u_nC_n)P_n^{\sigma^{\ast},\sigma}.
 	\end{eqnarray}
 	It follows from Lemma \ref{lem:frac-non-symm-jacobi}  that 
 	\begin{eqnarray}
 	\mathcal{L}_{\theta}^{\alpha}u
 	&=&\mathcal{L}_{\theta}^{\alpha}(\omega^{\sigma,\sigma^{\ast}}\sum_{n=0}^{\infty}u_nP_n^{\sigma-1,\sigma^{\ast}-1}) \nonumber\\
 	&=&\sum_{n=0}^{\infty}u_n(A_n\lambda_{n-2}P_{n-2}^{\sigma^{\ast},\sigma}+B_n\lambda_{n-1}P_{n-1}^{\sigma^{\ast},\sigma}+C_n\lambda_{n}P_{n}^{\sigma^{\ast},\sigma})\nonumber\\
 	&=& \sum_{n=0}^{\infty}\lambda_n(u_{n+2}A_{n+2}+u_{n+1}B_{n+1}+u_nC_n)P_n^{\sigma^{\ast},\sigma}.\label{spectral-eq-1}
 	\end{eqnarray}
 	From \eqref{eq:spectral-different-index-1}, we have 
 	\begin{eqnarray}
 	\sum_{n=0}^{\infty}f_nP_n^{\sigma^{\ast},\sigma}&=&\sum_{n=0}^{\infty}f_n(A_nP_{n-2}^{\sigma^{\ast},\sigma}-B_nP_{n-1}^{\sigma^{\ast},\sigma}+C_nP_{n}^{\sigma^{\ast},\sigma}) \nonumber \\
 	&=&\sum_{n=0}^{\infty}(f_{n+2}A_{n+2}-f_{n+1}B_{n+1}+f_nC_n)P_{n}^{\sigma^{\ast},\sigma}.
 	\label{spectral-eq-2}
 	\end{eqnarray}
 	Substituting \eqref{spectral-eq-1} and \eqref{spectral-eq-2} into \eqref{eq:two-sided} leads to 
 	\begin{eqnarray}\label{spectral-eq}
 &&	\sum_{n=0}^{\infty}\lambda_n(u_{n+2}A_{n+2}+u_{n+1}B_{n+1}+u_nC_n)P_n^{\sigma^{\ast},\sigma} \nonumber\\
 	&=&\sum_{n=0}^{\infty}(f_{n+2}A_{n+2}-f_{n+1}B_{n+1}+f_nC_n)P_{n}^{\sigma^{\ast},\sigma}.
 	\end{eqnarray}
 	Multiplying by $P_n^{\sigma^{\ast},\sigma}$ over both sides of the last equation \eqref{spectral-eq} and by the orthogonality of the Jacobi polynomials, we arrive at 
 	\begin{eqnarray}
 	&&u_{n+2}A_{n+2}+u_{n+1}B_{n+1}+u_nC_n\nonumber\\
 	&=&\frac{1}{\lambda_n}(f_{n+2}A_{n+2}-f_{n+1}B_{n+1}+f_nC_n)\nonumber  \\
 	&=&F_{n+2}A_{n+2}+F_{n+1}B_{n+1}+F_nC_n +(\frac{1}{\lambda_n}-\frac{1}{\lambda_{n+2}})f_{n+2}A_{n+2}\nonumber \\
 	&& -(\frac{1}{\lambda_n}+\frac{1}{\lambda_{n+1}})f_{n+1}B_{n+1} ,
 	\end{eqnarray}
 	where $F_n=f_n/\lambda_n.$
 	Using \eqref{eq:spectral-a1} again, we obtain 
 	\begin{eqnarray}
 	\sum_{n=0}^{\infty}u_nP_n^{\sigma-1,\sigma^{\ast}-1}&=&\sum_{n=0}^{\infty}F_nP_n^{\sigma-1,\sigma^{\ast}-1}+\sum_{n=0}^{\infty}(\frac{1}{\lambda_n}-\frac{1}{\lambda_{n+2}})f_{n+2}A_{n+2}P_n^{\sigma,\sigma^{\ast}}\nonumber\\
 	&&-\sum_{n=0}^{\infty}(\frac{1}{\lambda_n}+\frac{1}{\lambda_{n+1}})f_{n+1}B_{n+1}P_n^{\sigma,\sigma^{\ast}}.
 	\end{eqnarray}
 	Thus we have
 	\begin{eqnarray*}
 	u_n&=&F_n+\sum_{k=n}^{\infty}(\frac{1}{\lambda_k}-\frac{1}{\lambda_{k+2}})f_{k+2}A_{k+2}(h_n^{\sigma-1,\sigma^{\ast}-1})^{-1}(P_k^{\sigma,\sigma^{\ast}},P_n^{\sigma-1,\sigma^{\ast}-1})_{\omega^{\sigma-1,\sigma^{\ast}-1}}\nonumber\\
 	&&-\sum_{k=n}^{\infty}(\frac{1}{\lambda_k}+\frac{1}{\lambda_{k+1}})f_{k+1}B_{k+1} (h_n^{\sigma-1,\sigma^{\ast}-1})^{-1}(P_k^{\sigma,\sigma^{\ast}},P_n^{\sigma-1,\sigma^{\ast}-1})_{\omega^{\sigma-1,\sigma^{\ast}-1}},
 	\end{eqnarray*}
 	where $h_n^{\sigma-1,\sigma^{\ast}-1}$ is defined in
    \eqref{eq:orthogonality-of-jacobi-poly-normalized-constant}.
    Notice that $|A_k|\leq C$ and $|B_k|\leq C/k$ in Corollary
    \ref{cor:different-index}. 
 	By Lemma \ref{lem:integral-estimate},  we have 
 	\begin{eqnarray}
 	|u_n|\leq \frac{|f_n|}{\lambda_n}+C\sum_{k=n}^{\infty}(\frac{1}{\lambda_k}-\frac{1}{\lambda_{k+2}})|f_{k+2}|+C\sum_{k=n}^{\infty}(\frac{1}{\lambda_k}+\frac{1}{\lambda_{k+1}})\frac{1}{k}|f_{k+1}|.
 	\end{eqnarray}
 	When $r$ is an integer,  using the Cauchy-Schwarz inequality, we get 
 	\begin{eqnarray}
 	&&\sum_{k=n}^{\infty}(\frac{1}{\lambda_k}-\frac{1}{\lambda_{k+2}})|f_{k+2}|\nonumber \\
 	& \leq& \biggl[\sum_{k=n}^{\infty}(|f_{k+2}|d_{k+2,r}^{\sigma^{\ast}-1,\sigma-1})^2h_{k-r}^{\sigma^{\ast}-1+r,\sigma-1+r} \biggr]^{1/2} \nonumber \\
 	 && \times\biggl[\sum_{k=n}^{\infty}(\frac{1}{\lambda_k}-\frac{1}{\lambda_{k+2}})^2(d_{k+2,r}^{\sigma^{\ast}-1,\sigma-1})^{-2}(h_{k-r}^{\sigma^{\ast}-1+r,\sigma-1+r})^{-1}\biggr]^{1/2}\nonumber\\
 	&\leq& |f|_{B^r_{\omega^{\sigma^{\ast}-1,\sigma-1}}}\biggl[\sum_{k=n}^{\infty}(\frac{1}{\lambda_k}-\frac{1}{\lambda_{k+2}})^2(d_{k+2,r}^{\sigma^{\ast}-1,\sigma-1})^{-2}(h_{k-r}^{\sigma^{\ast}-1+r,\sigma-1+r})^{-1}\biggr]^{1/2}\nonumber\\
 	&\leq&C|f|_{B^r_{\omega^{\sigma^{\ast}-1,\sigma-1}}} n^{-r-\alpha}.
 	\end{eqnarray}
 	Similarly, we obtain  
 	\begin{equation}
 	\sum_{k=n}^{\infty}(\frac{1}{\lambda_k}+\frac{1}{\lambda_{k+1}})\frac{1}{k}|f_{k+1}|\leq C|f|_{B^r_{\omega^{\sigma^{\ast}-1,\sigma-1}}} n^{-r-\alpha}.
 	\end{equation}
 	Therefor, we have 
 	\begin{equation}
 	|u_n|\leq (\lambda_{\theta,n}^{\alpha})^{-1}|f_n|+Cn^{-r-\alpha}.
 	\end{equation}
 	Similar to the proof of Theorem \ref{thm:fode-regularity-flap-nonsymm} , we have $\omega^{-\sigma,-\sigma^{\ast}}u\in B^{r+\alpha}_{\omega^{\sigma-1,\sigma^{\ast}-1}}.$  When $r$ is not an integer, we can apply standard space interpolation to obtain the conclusion. 
 	
 	When $\mu\neq 0,$ we apply the bootstrapping technique. By $\mathcal{L}_{\theta}^{\alpha}u=f-\mu u\in B^{r\wedge 0}_{\omega^{\sigma-1,\sigma^{\ast}-1}}$ and the conclusion above, we have $\omega^{-\sigma,-\sigma^{\ast}}u\in B^{r\wedge 0+\alpha}_{\omega^{\sigma-1,\sigma^{\ast}-1}},$ which leads to $u\in B^{\alpha}_{\omega^{\sigma^{\ast}-1,\sigma-1}} .$ Thus  $\mathcal{L}_{\theta}^{\alpha}u=f-\lambda u\in B^{r\wedge \alpha}_{\omega^{\sigma^{\ast}-1,\sigma-1}}$  and by the conclusion above, we have $\omega^{-\sigma,-\sigma^{\ast}}u\in B^{r\wedge \alpha+\alpha}_{\omega^{\sigma-1,\sigma^{\ast}-1}},$ which further leads to $ u\in B^{\alpha+1}_{\omega^{\sigma^{\ast},\sigma}}$. Therefore, we have $\mathcal{L}_{\theta}^{\alpha}u=f-\mu u\in B^{r  \wedge(\alpha+1)}_{\omega^{\sigma^{\ast},\sigma}}.$  Then by Theorem \ref{thm:fode-regularity-flap-nonsymm}, we have $\omega^{-\sigma,-\sigma^{\ast}}u\in B^{r\wedge (\alpha+1)+\alpha}_{\omega^{\sigma,\sigma^{\ast}}}.$  	  
 \end{proof}

 
 \section{Spectral Galerkin method  and error analysis }

 
 The weak formulation is to find $u\in H_0^{\alpha/2},$ such that 
 \begin{eqnarray}\label{spectral-bilinear}
 (\mathcal{L}_{\theta}^{\alpha}u,v)+\mu (u,v)=(f,v), \quad \forall v \in H_0^{\alpha/2},
 \end{eqnarray} 
 where $H_0^{\alpha/2}$ is the standard fractional Sobolev space in \cite{Adams75} with induced semi-norm $|\cdot|_{H^{\alpha/2}},$ and $ (H_0^{\alpha/2})^{\prime}$ is the dual space of $ H_0^{\alpha/2}.$
 From \cite{ErvinR06}, we know that there exits a unique solution $u\in H_0^{\alpha/2}$ such that $\|u\|_{H^{\alpha/2}} \leq \|f\|_{(H_0^{\alpha/2})^{\prime}}. $ 
 
 We first consider the implementation of the spectral Galerkin method. Define the finite dimensional space 
 \[ U_N :=\omega^{\sigma,\sigma^{\ast}}{\mathbb{P}_N}=\text{Span}\{\phi_0,\phi_1,\ldots, \phi_N \},\] where 
$ \phi_k(x):=(1-x)^{\sigma}(1+x)^{\sigma^{\ast}}P^{\sigma,\sigma^{\ast}}_k(x),$ for $0\leq k\leq N$ and  $\mathbb{P}_N$ is the set of all algebraic   polynomials of degree at most $N.$ 
The spectral Galerkin 
 method is to   find $u_N\in U_N$ such that 
 \begin{eqnarray}\label{weak-form-scheme_3}
 (\mathcal{L}_{\theta}^{\alpha}u_N,v_N)+\mu (u_N,v_N)=(f,v_N),\quad \forall v_N\in U_N.
 \end{eqnarray}
 Plugging $u_N=\sum_{n=0}^N \hat{u}_n \phi_n(x)$ into  \eqref{weak-form-scheme_3} and taking $v_N=\phi_k(x)$,  we obtain from the orthogonality of Jacobi polynomials and Lemma \ref{lem:frac-non-symm-jacobi} that
 \begin{equation}
 \sum_{n=0}^N  S_{k,n}  \hat{u}_n +\mu\sum_{n=0}^N  M_{k,n}  \hat{u}_n
 =   (f, \phi_k),\quad k=0,1,2,\cdots,N,
 \end{equation}
 where 
 \begin{eqnarray}
 &&S_{k,n} = \lambda_{\theta,k}^{\alpha}\int_{-1}^{1}
 (1-x)^{\sigma}(1+x)^{\sigma^{\ast}}P_n^{\sigma^{\ast},\,\sigma}(x) P_k^{\sigma,\,\sigma^{\ast}}(x) \,dx,    \\
 && M_{k,n} =\int_{-1}^{1}
 (1-x)^{2\sigma}(1+x)^{2\sigma^{\ast}}P_n^{\sigma,\,\sigma^{\ast}}(x)P_k^{\sigma,\,\sigma^{\ast}}(x) \,dx .
 \end{eqnarray}
 To find $S_{k,n}$ and $M_{k,n}$,   we  apply    Gauss-Jacobi quadrature rule, e.g.
 \begin{eqnarray*}
 M_{k,n}=\int_{-1}^{1}  (1-x)^{2\sigma}(1+x)^{2\sigma^{\ast}}P_n^{\sigma,\,\sigma^{\ast}}(x)P_k^{\sigma^{\ast},\,\sigma}(x) \,dx= \sum_{j=0}^{N}P_n^{\sigma,\,\sigma^{\ast}}(x_j)P_k^{\sigma^{\ast},\,\sigma}(x_j)w_j.
 \end{eqnarray*}
 Here    $x_j$'s are the zeros of Jacobi polynomial $P_{N+1}^{2\sigma,2\sigma^{\ast}}(x)$, $w_j$'s are the corresponding quadrature weights.   The quadrature rule here is exact since $n+k\leq 2N$ while the quadrature rule  is exact for all $(2N+1)$-th order polynomials. The integral in  $S_{n,k}$ can be calculated similarly. 
 To find $f_k = (f, \phi_k)$, we use a different Gauss-Jacobi quadrature rule:
 $ f_k \approx \tilde{f}_k=\sum\limits_{j=0}^\mathsf{N} f(x_j)P_k^{\sigma^{\ast},\,\sigma}(x_j)w_j. $
 Here  $x_j$'s are the roots of Jacobi polynomial $P_{\mathsf{N}+1}^{\sigma^{\ast},\,\sigma}(x)$, $w_j$'s are the corresponding quadrature weights. 
 
  Next we focus on the analysis of the spectral Galerkin method. 
 
 
 To show the stability of \eqref{weak-form-scheme_3}, we require the following lemma.
 \begin{lemma}\cite{ErvinR06}
 	For any $v\in H_0^{\alpha/2}$ with $1<\alpha\leq 2,$ we  have 
 	\begin{equation}\label{frac-norm-equi}
 	(\mathcal{L}_{\theta}^{\alpha}v,v)=c_1^{\alpha}|v|^2_{H^{\alpha/2}},\quad c_1^{\alpha}=-\cos(\alpha/2\pi).
 	\end{equation} 
 \end{lemma}
 Since $U_N\subseteq H_0^{\alpha/2}$,   the well-posdeness of the discrete problem \eqref{weak-form-scheme_3} can be readily shown by the Lax-Milgram theorem.

 Before presenting the error estimates, we   need the following approximation properties.   Define    the $L_{\omega^{\sigma,\sigma^{\ast}}}^2$-orthogonal projection $\mathcal{P}_N^{\sigma,\sigma^{\ast}}:L_{\omega^{\sigma,\sigma^{\ast}}}^2\rightarrow \mathbb{P}_N $  such that $(\mathcal{P}_N^{\sigma,\sigma^{\ast}}u-u,v )_{\omega^{\sigma,\sigma^{\ast}}}=0$ for any $v\in \mathbb{P}_N,$ 
 or equivalently $\mathcal{P}_N^{\sigma,\sigma^{\ast}}(u)(x)=\sum_{n=0}^N\hat{u}_n P_n^{\sigma,\sigma^{\ast}}(x)$.
 Denote $\Pi_N^{\sigma,\sigma^{\ast}}u:=\omega^{\sigma,\sigma^{\ast}}\mathcal{P}_N^{\sigma,\sigma^{\ast}}(\omega^{-\sigma,-\sigma^{\ast}}u)$.
 \begin{lemma}\label{lem:approximation-property}
 	Let $\omega ^{-\sigma,-\sigma^{\ast}}u \in B^{m}_{ \omega^{\sigma,\sigma^{\ast}} }$ and	 $\mathcal{L}_{\theta}^{\alpha}u \in B^{m}_{ \omega^{\sigma^{\ast},\sigma} }$.  Then, for $0\leq m\leq N,$ we have the following estimates 
 	\begin{eqnarray}\label{error-weighted-L2}
 	\|u- \Pi_N^{\sigma,\sigma^{\ast}}u\|_{\omega^{-\sigma,-\sigma^{\ast}}}\leq cN^{-m}| \omega ^{-\sigma,-\sigma^{\ast}}u|_{B^{m}_{\omega^{\sigma,\sigma^{\ast}}}}
 	\end{eqnarray}
 	and 
 	\begin{eqnarray}\label{error-fractional-weighted}
 	\|\mathcal{L}_{\theta}^{\alpha}(u- \Pi_N^{\sigma,\sigma^{\ast}}u)\|_{\omega^{\sigma^{\ast},\sigma}} \leq cN^{-m}|\mathcal{L}_{\theta}^{\alpha}u|_{B^m_{\omega^{\sigma^{\ast},\sigma}}}.
 	\end{eqnarray}
 \end{lemma}
 \begin{proof}
 	For $\tilde{u}=\omega ^{-\sigma,-\sigma^{\ast}}u\in B^{m}_{ \omega^{\sigma,\sigma^{\ast}} }(I),$ we have the expansion as 
 	\begin{eqnarray}\label{L2expansion}
 	u=\omega^{\sigma,\sigma^{\ast}}\tilde{u}=\omega^{\sigma,\sigma^{\ast}} \sum_{n=0}^{\infty}{u}_nP_n^{\sigma,\sigma^{\ast}}.
 	\end{eqnarray}
 Noticing \eqref{L2expansion}, we obtain  
 	\begin{eqnarray}
 	&&u-\Pi_N^{\sigma,\sigma^{\ast}}u =\omega^{\sigma,\sigma^{\ast}}\sum_{n=N+1}^{\infty}{u}_nP_n^{\sigma,\sigma^{\ast}} =  \omega^{\sigma,\sigma^{\ast}} (\tilde{u} -\mathcal{P}_N^{\sigma,\sigma^*}\tilde{u}).
 	\end{eqnarray}
 	By  Lemma \ref{lem:frac-non-symm-jacobi}, it holds that 
 	\[ \mathcal{L}_{\theta}^{\alpha}(u-\Pi_N^{\sigma,\sigma^{\ast}}u) =\sum_{n=N+1}^{\infty}\lambda_{\theta,n}^{\alpha}{u}_nP_n^{\sigma^{\ast},\sigma}=\mathcal{L}_{\theta}^{\alpha}u-\Pi_N^{\sigma^{\ast},\sigma} (\mathcal{L}_{\theta}^{\alpha}u). \]
 		Then by the error estimate for the orthogonal projection  $\mathcal{P}_N^{\sigma,\sigma^*}$, e.g. in \cite{MaSun01},   we reach the conclusion.    	  
 \end{proof}

  In order to show convergence,   we also need Hardy-type inequality below.
 \begin{lemma}\label{lem:frac-hardy-ineq} \cite{LossS10}
 	Let $\Omega$ be a convex set and $1<\alpha<2.$  For any
    $v\in H^{\alpha/2}$   vanishing  on the boundary, it holds 
 	\begin{equation}
 |v|^2_{H^{\alpha/2}}  \geq C \int_{\Omega\times\Omega}\frac{|v(x)-v(y)|^2}{|x-y|^{n+\alpha}}dxdy	 \geq k_{n,\alpha}\int_{\Omega}\frac{|v(x)|^2}{d_{\Omega}(x)^{\alpha}}dx, 	\end{equation}
 	where $C$ and $k_{n,\alpha}$ are positive constants which only depend  on dimension $n$ and $\alpha$,  and $d_{\Omega}(x)$ denotes the distance from the point $x\in \Omega$ to the boundary of the $\Omega$.
 \end{lemma}
 
With above Lemma, we can obtain the following result. 
 \begin{lemma}
 	For any $v\in H_0^{\alpha/2}$ with $1<\alpha\leq 2,$ we  have 
 	\begin{equation}\label{frac-hardy-ineq} 
 	\|v\|_{\omega ^{-\sigma ,-\sigma^{\ast}} }\leq C 
 	|v|_{H^{\alpha/2}},\quad 
 	\|v\|_{\omega ^{-\sigma^{\ast} ,-\sigma} }\leq C 
 	|v|_{H^{\alpha/2}}.
 	\end{equation}
 \end{lemma}
 \begin{proof} 
 	The fractional Hardy-type inequality in Lemma \ref{lem:frac-hardy-ineq} 	 leads to 
 	\begin{equation}\label{sp-ineq-1}
 	\|v\|_{\omega ^{- \sigma ,-\sigma^{\ast}}} \leq c   \|v \omega^{-\alpha/2}\| \leq  C 	 |v|_{H^{\alpha/2}} .
 	\end{equation}
    The second inequality follows similarly.  	  
 \end{proof}
 
 \begin{theorem}\label{thm:spectral-error-estimate-1}
 	Suppose that $u$ and $u_N$ satisfy the problems \eqref{spectral-bilinear} and  \eqref{weak-form-scheme_3}, respectively. 
 	If  $f\in B^r_{ \omega^{\sigma^{\ast}-1,\sigma-1}}$ with  $r\geq 0$, we have the following  error estimates:
 	\begin{eqnarray}
 	\|u-u_N\|_{\omega^{-\sigma,-\sigma^{\ast}}}\leq  C N^{-\gamma}\abs{ \omega ^{-\sigma,-\sigma^{\ast}}u}_{B^{\gamma  }_{\omega^{\sigma,\sigma^{\ast}}}}, \quad \gamma=(\alpha+1)\wedge r+ \alpha
 	\end{eqnarray}
 	when $\theta=0.5$ and 
 	\begin{eqnarray}
 	\|u-u_N\|_{\omega^{-\sigma,-\sigma^{\ast}}}&\leq&   cN^{\alpha-\gamma}\abs{ \mathcal{L}_{\theta}^{\alpha}u}_{B^{\gamma-\alpha}_{\omega^{\sigma^{\ast},\sigma}}} + CN^{-\gamma}\abs{ \omega ^{-\sigma,-\sigma^{\ast}}u}_{B^{\gamma}_{\omega^{\sigma,\sigma^{\ast}}}}
 	\end{eqnarray}
 	when	$\theta\neq 0.5.$
 \end{theorem}
 \begin{proof}
 	Denote $\eta_N=u-\Pi_N^{\sigma,\sigma^{\ast}}u$ and $e_N=\Pi_N^{\sigma,\sigma^{\ast}}u-u_N,$ then $u-u_N=\eta_N+e_N.$
 	Combining  \eqref{spectral-bilinear} and \eqref{weak-form-scheme_3}, we can obtain the  following error equation:
 	\begin{eqnarray*}
 		(\mathcal{L}_{\theta}^{\alpha}e_N,v_N)+\mu (e_N,v_N)=-  (\mathcal{L}_{\theta}^{\alpha}\eta_N,v_N)-\mu (\eta_N,v_N),\quad  \forall v_N\in U_N.
 	\end{eqnarray*}
 	Taking $v_N=e_N$ leads to 
 	\begin{eqnarray}
 	(\mathcal{L}_{\theta}^{\alpha}e_N,e_N)+\mu \|e_N\|^2  &=&-  (\mathcal{L}_{\theta}^{\alpha}\eta_N,e_N)-\mu (\eta_N,e_N) .\label{spectral-eq-3} 
 	\end{eqnarray}
 	
 	When $\theta=0.5$,  $\sigma=\sigma^{\ast}=\alpha/2$ and    the orthogonal property leads to 	\[ (\mathcal{L}_{0.5}^{\alpha}(u-\Pi_N^{\alpha/2,\alpha/2}u),v_N ) =0,\quad \forall v_N\in U_N= \omega^{\alpha/2,\alpha/2}\mathbb{P}_N.\]
 Thus, when $\theta=0.5$, \eqref{spectral-eq-3} reduces to 
 \begin{equation}
 	(\mathcal{L}_{0.5}^{\alpha}e_N,e_N)+\mu \|e_N\|^2 =-\mu (\eta_N,e_N). \label{spectral-eq-4}
 \end{equation}	
 By the  	norm equivalence \eqref{frac-norm-equi} and fractional Hardy-type inequality 	\eqref{frac-hardy-ineq} that
 	\begin{eqnarray}\label{spectral-ineq-1}
 	(\mathcal{L}_{0.5}^{\alpha}e_N,e_N)+\mu \|e_N\|^2\geq C \|e_N\|^2_{\omega^{-\sigma,-\sigma^{\ast}}}.
 	\end{eqnarray}
 	 Combining 
 	\eqref{spectral-eq-4} and \eqref{spectral-ineq-1}, we have 
 	\begin{eqnarray}\label{error}
 	\|e_N\|^2_{\omega^{-\sigma,-\sigma^{\ast}}} \leq  C\|\eta_N\| \|e_N\|  \leq C \|\eta_N\|_{\omega^{-\sigma,-\sigma^{\ast}}} \|e_N\|_{\omega^{-\sigma,-\sigma^{\ast}}} .
 	\end{eqnarray}
 	Thus  $ \|e_N\|_{\omega^{-\sigma,-\sigma^{\ast}}}\leq C\|\eta_N\|_{\omega^{-\sigma,-\sigma^{\ast}}}.$
 	By the regularity in Theorem
    \ref{thm:fode-regularity-two-side} and the projection
    estimate  in Lemma  \ref{lem:approximation-property}, we arrive at
 	\begin{eqnarray*}\label{}
 		\|e_N\|_{\omega^{-\sigma,-\sigma^{\ast}}}
 		&\leq&   CN^{-\gamma}\abs{ \omega ^{-\sigma,-\sigma^{\ast}}u}_{B^{\gamma}_{\omega^{\sigma,\sigma^{\ast}}}},\quad   \theta =0.5.
 	\end{eqnarray*}
 	
 	We now  turn to the case $ \theta \neq 0.5$.  	 
 	Using the norm equivalence \eqref{frac-norm-equi} and the fractional Hardy-type inequality 	\eqref{frac-hardy-ineq}, we have 
 	\begin{eqnarray*}\label{}
 	c_1^{\alpha}|e_N|^2_{H^{{\alpha}/{2}}}+\mu \|e_N\|^2
 	&\leq & \| \mathcal{L}_{\theta}^{\alpha}\eta_N\|_{\omega^{\sigma^{\ast},\sigma}}\|e_N\|_{\omega^{-\sigma^{\ast},-\sigma}}+\mu \|\eta_N\| \|e_N\|\nonumber\\
 	&\leq& \| \mathcal{L}_{\theta}^{\alpha}\eta_N\|_{\omega^{\sigma^{\ast},\sigma}} \cdot  c_2^{\alpha} 
 	|e_N|_{H^{\alpha/2}}+\mu \|\eta_N\| \|e_N\| \nonumber\\
 	&\leq & \frac{(c_2^{\alpha} )^2} {2c_1^{\alpha}}    \| \mathcal{L}_{\theta}^{\alpha}\eta_N\|^2_{\omega^{\sigma^{\ast},\sigma}} + \frac{c_1^{\alpha}}{2}  
 	|e_N|^2_{H^{\alpha/2}} +\frac{\mu}{2} \|\eta_N\|^2 +\frac{\mu}{2} \|e_N\|^2.
 	\end{eqnarray*}	
 	Thus it follows that 
 	\begin{eqnarray*}\label{}
 	C \|e_N\|^2_{\omega^{-\sigma,-\sigma^{\ast}}}\leq c_1^{\alpha}|e_N|^2_{H^{{\alpha}/{2}}}+\mu \|e_N\|^2
 	\leq  \frac{(c_2^{\alpha} )^2} {c_1^{\alpha}}    \| \mathcal{L}_{\theta}^{\alpha}\eta_N\|^2_{\omega^{\sigma^{\ast},\sigma}} + c \|\eta_N\|_{\omega^{-\sigma,-\sigma^{\ast}}}^2 .
 	\end{eqnarray*}
 	From  Theorem  \ref{thm:fode-regularity-two-side}  and Lemma  \ref{lem:approximation-property}, we  obtain
 	\begin{eqnarray*}\label{}
 		\|e_N\|_{\omega^{-\sigma,-\sigma^{\ast}}}
 		&\leq&  cN^{\alpha-\gamma}}\|{\mathcal{L}_{\theta}^{\alpha}u\|_{B^{\gamma-\alpha}_{\omega^{\sigma^{\ast},\sigma}}} + CN^{-\gamma}\| \omega ^{-\sigma,-\sigma^{\ast}}u\|_{B^{\gamma}_{\omega^{\sigma,\sigma^{\ast}}}}.
 	\end{eqnarray*}
For both cases $\theta=0.5$ and $\theta\neq 0.5$, 
we arrive at the desired results by  the triangle inequality $\|u-u_N\|_{\omega^{-\sigma,-\sigma^{\ast}}}\leq \|e_N\|_{\omega^{-\sigma,-\sigma^{\ast}}}+\|\eta_N\|_{\omega^{-\sigma,-\sigma^{\ast}}}$. 
 \end{proof}	
By Theorems \ref{thm:fode-regularity-flap-nonsymm} and \ref{thm:spectral-error-estimate-1},  it is straightforward to obtain the following result. 
  \begin{corollary}\label{cor:spectral-error-estimate-1}
  	Suppose that $u$ and $u_N$ satisfy the problems \eqref{spectral-bilinear} and  \eqref{weak-form-scheme_3}, respectively. 
  	If  $f\in B^r_{ \omega^{\sigma^{\ast},\sigma}}$ and $u\in L^\infty$,  we have the following  error estimates:
  	\begin{eqnarray}
  	\|u-u_N\|_{\omega^{-\sigma,-\sigma^{\ast}}}\leq  C N^{-\gamma}\abs{ \omega ^{-\sigma,-\sigma^{\ast}}u}_{B^{\gamma  }_{\omega^{\sigma,\sigma^{\ast}}}}, \quad \gamma=\alpha\wedge r+ \alpha
  	\end{eqnarray}
  	when $\theta=0.5$ and 
  	\begin{eqnarray}
  	\|u-u_N\|_{\omega^{-\sigma,-\sigma^{\ast}}}&\leq&   cN^{\alpha-\gamma}\abs{ \mathcal{L}_{\theta}^{\alpha}u}_{B^{\gamma-\alpha}_{\omega^{\sigma^{\ast},\sigma}}} + CN^{-\gamma}\abs{ \omega ^{-\sigma,-\sigma^{\ast}}u}_{B^{\gamma}_{\omega^{\sigma,\sigma^{\ast}}}}
  	\end{eqnarray}
  	when	$\theta\neq 0.5.$
  \end{corollary}
 
 \section{Spectral Petrov-Galerkin method and error analysis}
 
The spectral Petrov-Galerkin 
 method is to   find $u_N\in U_N$ such that 
 \begin{eqnarray}\label{PG-scheme}
 (\mathcal{L}_{\theta}^{\alpha}u_N,v_N)+\mu (u_N,v_N)=(f,v_N),\quad \forall v_N\in V_N.
 \end{eqnarray}
The method is implicitly discussed in \cite{ErvinHR16} and is
fully discussed in \cite{MaoKar17} when $mu=0$.  
Here  we define the finite dimensional space $ V_N:= \omega^{\sigma^{\ast},\sigma}{\mathbb{P}_N}=\text{Span}\{\varphi_0,\varphi_1,\ldots, \varphi_N \}$ and 
 $
 \varphi_k(x):=(1-x)^{\sigma^{\ast}}(1+x)^{\sigma}P^{\sigma^{\ast},\sigma}_k(x)$.
 
 For  implementation,    plugging $u_N=\sum_{n=0}^N \hat{u}_n \phi_n(x)$ in  \eqref{PG-scheme} and taking $v_N=\varphi_k(x)$,  we obtain from   Lemma \ref{lem:frac-non-symm-jacobi}  and the orthogonality of Jacobi polynomials that
 \begin{equation}\label{PG-scheme-2}
 \lambda_{\theta,k}^{\alpha}h_k^{\sigma^{\ast},\sigma}\hat{u}_k +\mu\sum_{n=0}^N  M_{k,n}  \hat{u}_n
 =   (f, \varphi_k),\quad k=0,1,2,\cdots,N,
 \end{equation}
 where $   \lambda_{\theta,k}^{\alpha}$ is defined in Lemma \ref{lem:frac-non-symm-jacobi} and
 \begin{eqnarray}
 M_{k,n} =\int_{-1}^{1}
 (1-x^2)^{\alpha}P_n^{\sigma,\,\sigma^{\ast}}(x)P_k^{\sigma^{\ast},\,\sigma}(x) \,dx .
 \end{eqnarray}
 Here $M_{k,n}$ and $f_k = (f, \varphi_k)$ can be computed
 similarly  as in  the last section.

 Now we turn to the analysis of the spectral Petrov-Galerkin method. 
 
 \begin{lemma}\label{sp-lem-emb}
 	Let $\alpha\in (1,2).$ Suppose   $u$ satisfies   $ \omega^{-\sigma,-\sigma^{\ast}} u\in L^2_{\omega^{\sigma,\sigma^{\ast}}}$ and  $\|\mathcal{L}_{\theta}^{\alpha}u\|_{\omega^{\sigma^{\ast},\sigma}} <\infty$. Then we have 
 	\begin{eqnarray}
 	\lambda_{\theta,0}^{\alpha} \|u\|_{\omega^{\sigma^{\ast},\sigma}} \leq\lambda_{\theta,0}^{\alpha} \|u\|_{\omega^{-\sigma,-\sigma^{\ast}}}   \leq 	\|\mathcal{L}^{\alpha}_{\theta}u\|_{\omega^{\sigma^{\ast},\sigma}}  .
 	\end{eqnarray}	
 	
 \end{lemma}
 \begin{proof}
 	For $u$  satisfying      $ \omega^{-\sigma,-\sigma^{\ast}} u\in L^2_{\omega^{\sigma,\sigma^{\ast}}},$  we write 
 	\begin{eqnarray}\label{L2weighted-expansion}
 	u=\omega^{\sigma,\sigma^{\ast}} \sum_{n=0}^{\infty}{u}_nP_n^{\sigma,\sigma^{\ast}}.
 	\end{eqnarray}
 	Given the expansion \eqref{L2weighted-expansion}, we derive from Lemma \ref{lem:frac-non-symm-jacobi}  that 
 	\begin{eqnarray*}
 		&&\|u\|^2_{\omega^{-\sigma,-\sigma^{\ast}}}=\sum_{n=0}^{\infty}|{u}_n|^2h_n^{\sigma,\sigma^{\ast}},\quad\|\mathcal{L}^{\alpha}_{\theta}u\|^2_{\omega^{\sigma^{\ast},\sigma}}=\sum_{n=0}^{\infty}(\lambda_{\theta,n}^{\alpha})^2|{u}_n|^2h^{\sigma^{\ast},\sigma}_n,
 	\end{eqnarray*}
 	where by \eqref{eq:orthogonality-of-jacobi-poly-normalized-constant}, we have
 	$h_{n}^{\sigma,\sigma^{\ast}}=h_{n}^{\sigma^{\ast},\sigma}$.
 	Noticing  the sequence $\{\lambda_{\theta,n}^{\alpha}\}$ is monotonically increasing,
 	we have 
 	\begin{eqnarray*}
 		\lambda_{\theta,0}^{\alpha} \|u\|_{\omega^{\sigma^{\ast},\sigma}} \leq\lambda_{\theta,0}^{\alpha} \|u\|_{\omega^{-\sigma,-\sigma^{\ast}}}   \leq 	\|\mathcal{L}^{\alpha}_{\theta}u\|_{\omega^{\sigma^{\ast},\sigma}}  .
 	\end{eqnarray*}
 	This completes the proof.  
 \end{proof}


 \begin{theorem}[Stability]
 	Assume that $|\mu|\leq \lambda_{\theta,0}^{\alpha}/2.$ 
 	The problem \eqref{PG-scheme} admits a unique solution $u_N$ such that 
 	$$	\|\mathcal{L}^{\alpha}_{\theta}u_N\|_{\omega^{\sigma^{\ast},\sigma}}\leq C\|f\| _{\omega^{\sigma^{\ast},\sigma}}. $$
 \end{theorem}
 
 \begin{proof}
 	Take $v_N=\omega^{\sigma^{\ast},\sigma}\mathcal{L}_{\theta}^{\alpha}u_N$ in \eqref{PG-scheme}. By Lemma \ref{sp-lem-emb}, we get 
 	\begin{eqnarray*}
 		\|\mathcal{L}_{\theta}^{\alpha}u_N\|^2_{\omega^{\sigma^{\ast},\sigma}} &=&-\mu (u_N,\omega^{\sigma^{\ast},\sigma}\mathcal{L}_{\theta}^{\alpha}u_N)+(f,\omega^{\sigma^{\ast},\sigma}\mathcal{L}_{\theta}^{\alpha}u_N)\\
 		&\leq & |\mu| \|u_N\|_{\omega^{\sigma^{\ast},\sigma}} \|\mathcal{L}_{\theta}^{\alpha}u_N\|_{\omega^{\sigma^{\ast},\sigma}}+  \|f\|_{\omega^{\sigma^{\ast},\sigma}} \|\mathcal{L}_{\theta}^{\alpha}u_N\|_{\omega^{\sigma^{\ast},\sigma}}\\
 		&\leq & |\mu|/\lambda_{\theta,0}^{\alpha}   \|\mathcal{L}_{\theta}^{\alpha}u_N\|^2_{\omega^{\sigma^{\ast},\sigma}}+  \|f\|_{\omega^{\sigma^{\ast},\sigma}} \|\mathcal{L}_{\theta}^{\alpha}u_N\|_{\omega^{\sigma^{\ast},\sigma}}\\
 		&\leq & 1/2 \|\mathcal{L}_{\theta}^{\alpha}u_N\|^2_{\omega^{\sigma^{\ast},\sigma}}+  \|f\|_{\omega^{\sigma^{\ast},\sigma}} \|\mathcal{L}_{\theta}^{\alpha}u_N\|_{\omega^{\sigma^{\ast},\sigma}},
 	\end{eqnarray*}
 	which leads to desired result directly.  
 \end{proof}

 \begin{theorem}[Convergence order]\label{thm:spectral-error-estimate-2}
 	Suppose that $u$ and $u_N$ satisfy  the problems \eqref{spectral-bilinear} and \eqref{PG-scheme}, respectively. If  $f\in  B^{r}_{ \omega^{\sigma^{\ast}-1,\sigma-1}}$ with $r\geq 0$ and  $|\mu|\leq \lambda_{\theta,0}^{\alpha}/2,$ then we have the following optimal error estimate:
 	\begin{eqnarray}
 	\|u-u_N\|_{\omega^{-\sigma,-\sigma^{\ast}}}\leq  cN^{-\gamma}| \omega ^{-\sigma,-\sigma^{\ast}}u|_{B^{\gamma}_{\omega^{\sigma,\sigma^{\ast}}}},\quad \gamma= (\alpha+1)\wedge r+ \alpha.
 	\end{eqnarray}
 \end{theorem}
 \begin{proof}
 	Denote $\eta_N=u-\Pi_N^{\sigma,\sigma^{\ast}}u$ and $e_N=\Pi_N^{\sigma,\sigma^{\ast}}u-u_N,$ then $u-u_N=\eta_N+e_N.$
 	Combining  \eqref{spectral-bilinear} and \eqref{PG-scheme}, we can obtain the  following error equation:
 	\begin{eqnarray*}
 		(\mathcal{L}_{\theta}^{\alpha}e_N,v_N)+\mu (e_N,v_N)=-  (\mathcal{L}_{\theta}^{\alpha}\eta_N,v_N)-\mu (\eta_N,v_N),\quad  \forall v_N\in U_N.
 	\end{eqnarray*}
 	Taking $v_N=\omega^{\sigma^{\ast},\sigma}\mathcal{L}_{\theta}^{\alpha}e_N,$ and by the orthogonal property, we get
 	\begin{eqnarray*}
 		\|\mathcal{L}_{\theta}^{\alpha}e_N\|^2_{\omega^{\sigma^{\ast},\sigma}} &=&-\mu (e_N,\omega^{\sigma^{\ast},\sigma}\mathcal{L}_{\theta}^{\alpha}e_N)-  (\mathcal{L}_{\theta}^{\alpha}\eta_N,\omega^{\sigma^{\ast},\sigma}\mathcal{L}_{\theta}^{\alpha}e_N)-\mu (\eta_N,\omega^{\sigma^{\ast},\sigma}\mathcal{L}_{\theta}^{\alpha}e_N) \nonumber \\
 		&=&-\mu (e_N,\omega^{\sigma^{\ast},\sigma}\mathcal{L}_{\theta}^{\alpha}e_N)- \mu (\eta_N,\omega^{\sigma^{\ast},\sigma}\mathcal{L}_{\theta}^{\alpha}e_N).	
 	\end{eqnarray*}
 	Following a similar derivation in the proof for stability,     we obtain 	 
 	$$ 	\|\mathcal{L}_{\theta}^{\alpha}e_N\|_{\omega^{\sigma^{\ast},\sigma}} \leq  2   \mu\|\eta_N\|_{\omega^{\sigma^{\ast},\sigma}}. $$ 
 	By Lemma \ref{sp-lem-emb}, we have   
 	$$\|e_N\|_{\omega^{-\sigma,-\sigma^{\ast}}} \leq  1/\lambda_{\theta,0}^{\alpha}  \|\mathcal{L}_{\theta}^{\alpha}e_N\|_{\omega^{\sigma^{\ast},\sigma}} \leq  2   \mu/\lambda_{\theta,0}^{\alpha}   \|\eta_N\|_{\omega^{\sigma^{\ast},\sigma}}  \leq    \|\eta_N\|_{\omega^{-\sigma,-\sigma^{\ast}}}. $$
 	Using the triangle inequality leads to 
 	$$	\|u-u_N\|_{\omega^{-\sigma,-\sigma^{\ast}}}\leq \|e_N\|_{\omega^{-\sigma,-\sigma^{\ast}}} +
 	\|\eta_N\|_{\omega^{-\sigma,-\sigma^{\ast}}}
 	\leq 2\|\eta_N\|_{\omega^{-\sigma,-\sigma^{\ast}}}. $$
 	Since $f\in  B^{r}_{ \omega^{\sigma^{\ast}-1,\sigma-1}}$  with $r\ge 0$, we can see from Theorem   \ref{thm:fode-regularity-two-side} that 	$\omega^{-\sigma,-\sigma^{\ast}}u \in   B^{ \gamma }_{\omega^{ \sigma,\sigma^{\ast} }}.$ Applying Lemma  \ref{lem:approximation-property}  leads to the desired result. 
 \end{proof}
 \begin{remark}
 	For the spectral Petrov-Galerkin   method, we need $|\mu|\leq \lambda_{\theta,0}^{\alpha}/2$. However, this assumption seems to be   relaxed to  the case for all $\mu>-\lambda_{\theta,0}^{\alpha}/2$. The key is to show that 
 	$(u_N,\omega^{\sigma^{\ast},\sigma}\mathcal{L}_{\theta}^{\alpha}u_N)$ is positive  for all $u_N\in U_N$. Unfortunately, we are not able to prove this for technical reasons. However,  when $\theta=0.5$, the spectral Petrov-Galerkin  method coincides with spectral Galerkin  method and we only need $\mu>-\lambda_{\theta,0}^{\alpha}/2$. 
 \end{remark}
 
 By Theorems \ref{thm:fode-regularity-flap-nonsymm} and \ref{thm:spectral-error-estimate-2},  it is straightforward to obtain the following result. 
  \begin{corollary}[Convergence order]\label{cor:spectral-error-estimate-2}
  	Suppose that $u$ and $u_N$ satisfy  the problems \eqref{spectral-bilinear} and \eqref{PG-scheme}, respectively. If  $f\in  B^{r}_{ \omega^{\sigma^{\ast},\sigma}}$, $u\in L^\infty$, and  $|\mu|\leq \lambda_{\theta,0}^{\alpha}/2,$ then we have the following optimal error estimates:
  	\begin{eqnarray}
  	\|u-u_N\|_{\omega^{-\sigma,-\sigma^{\ast}}}\leq  cN^{-\gamma}| \omega ^{-\sigma,-\sigma^{\ast}}u|_{B^{\gamma}_{\omega^{\sigma,\sigma^{\ast}}}},\quad \gamma= \alpha\wedge r+ \alpha.
  	\end{eqnarray}
  \end{corollary}

 

 \section{Numerical results} \label{sec:numerical example}
 In this section, we present three examples with  different   forcing terms $f$: 
 smooth, weakly singular at an interior point and  weakly singular at boundary. Since    exact solutions  are unavailable, we use    reference solutions $u_{\rm  ref}$, which are computed  with a very fine resolution using the same methods for computing $u_N$. 
In the  computation, we take $\mu=1$ and measure the error as follows
 \begin{eqnarray} \label{eq:error-1}
 &&E_1(N)=\|u_{\rm ref}-u_N\|_{\omega^{-2\sigma,-2\sigma^{\ast}}},\quad  E_2(N)=\|u_{\rm ref}-u_N\|_{\omega^{-\sigma,-\sigma^{\ast}}} 
 \end{eqnarray}
 where
 $ u_N=  \sum_{n=0}^{N}\hat{u}_n\omega^{\sigma,\sigma^{\ast}}P_n^{\sigma,\sigma^{\ast}}$ and $u_{\rm ref}=: u_{512}.$
   
 For the first two examples without boundary singularity, we use  $E_1(N)$ to measure the error. As $E_2(N)\leq E_{1}(N)$, the convergence order of $E_{1}(N)$ is at  least the order of $E_{2}(N)$. For the third example, we use $E_2(N)$  to measure the error.
  
Here we present numerical results for $\theta \in [0.5,1]$, in particular,    $\theta=0.5,\,0.7$, and    $\theta=1$.  Since $\sigma$ and $\sigma^{\ast}$ depends on the fractional order $\alpha$ and $\theta$, we find the values of  $\sigma,\,\sigma^{\ast}$ numerically using Newton's method with a tolerance $10^{-16}.$ 
 We list in  Table \ref{spectral-sigma-pair} the values of
 $(\sigma,\sigma^{\ast})$  for  different $\theta$'s and
 $\alpha$'s.  For illustration, we present only four digits in
 the table while in computation we keep fifteen digits for $\sigma$ and $\sigma^*$.

 \begin{table}[!ht]
 	\centering
 	\caption{  Numerical values for  $(\sigma, \sigma^{\ast})$ corresponding to different $\theta$ and $\alpha$  }\label{spectral-sigma-pair}
 	\begin{tabular}{ccccc}
 		\hline\hline
 		& $\alpha=1.2$& $\alpha=1.4$  & $\alpha=1.6$& $\alpha=1.8$ \\
 		
 		$\theta=0.5$ & (0.6000, 0.6000) & (0.7000, 0.7000)  & (0.8000, 0.8000) & (0.9000, 0.9000) \\
 		$\theta=0.7$ & (0.8829, 0.3171) & (0.8602, 0.5398)  & (0.8900, 0.7100) & (0.9411, 0.8589) \\
 		$\theta=1.0$ & (1.0000, 0.2000) & (1.0000, 0.4000)  & (1.0000, 0.6000) & (1.0000, 0.8000) \\
 		\hline
 		\hline
 	\end{tabular}
 \end{table}
 
  For spectral Petrov-Galerkin  method, we present   numerical results for 
  all $\theta$'s listed above while we present only numerical results of the Galerkin method for  $\theta\neq 0.5$ since the method is the same as the spectral Petrov-Galerkin method  when $\theta=0.5$.

 \begin{example}\label{spectral-ex-1}
 	Consider 	$f= \sin x.$  Here  $f$   belongs to $ B^{\infty}_{\omega^{\sigma^{\ast}-1,\sigma-1}}.$
 	By Theorem \ref{thm:fode-regularity-two-side},   	$\omega^{-\sigma,-\sigma^{\ast}}u \in   B^{ 2\alpha+1 }_{\omega^{ \sigma,\sigma^{\ast} }}$.
 \end{example}
 According to Theorem \ref{thm:spectral-error-estimate-2}, 
 the convergence orders are expected to be $2\alpha+1$ for the spectral Petrov-Galerkin   method \eqref{PG-scheme}. 
    In Table     \ref{table:spectral-PG-exam-1},  we observe that the convergence orders are $2\alpha+1$ for the spectral Petrov-Galerkin method   when  the order $\alpha=1.2,\,1.4,\,1.6,\,1.8$.

    In the light of  Theorem   \ref{thm:spectral-error-estimate-1}, 
    the convergence orders are expected to be  $\alpha+1$ for the spectral Galerkin  method \eqref{weak-form-scheme_3}. However, we found  numerically the convergence orders lie in between $\alpha+1$ and $2\alpha+1$; see Table \ref{table:spectral-G-exam-1}. 
   For $\theta=0.7$ in  Table \ref{table:spectral-G-exam-1}, the observed convergence orders are $2\alpha+1$ when   $\alpha=1.4,\,1.6,\,1.8,$ instead of $\alpha+1$.
     When $\alpha=1.2,$ the order is not close to $2\alpha+1$ but still  larger than   $\alpha+1$. 
      For  $\theta=1$, the convergence orders again lie in between  $\alpha+1$ and $2\alpha+1$ except for   $\alpha=1.2$.  For $\alpha=1.2,$ the convergence order is around 2 and is smaller than the expected $\alpha+1=2.2$. %
    However, the   order $2.2$ can be observed for $\alpha=1.2$ if we take  $u_{\rm ref}=u_{256}$ instead of $u_{\rm ref}=u_{512}.$ More precisely, the errors   are  $3.56e-04$ ($N=16$),  $9.56e-05$ ($N=32$), $2.44e-05$ ($N=64$),  $5.36e-06$ ($N=128$) and the computed orders are 1.90, 1.97, 2.19. 
  We also observe that when $\alpha=1.8$, the convergence order can match the regularity  index $2\alpha+1.$  
 
 In this example, the spectral Petrov-Galerkin  method \eqref{PG-scheme} has  
 the convergence order $2\alpha+1$, which suggests  that the regularity index $2\alpha+1$ for the solution and somewhat verifies Theorem \ref{thm:fode-regularity-two-side}.  The spectral Petrov-Galerkin  method  has higher accuracy  than the  spectral Galerkin   method  \eqref{weak-form-scheme_3} does, especially for small fractional order $\alpha$ when $\theta \neq 0.5$.  An improvement in the convergence order for the spectral Galerkin method will be considered for further research.
 
  \begin{table}[!htb]
  	\centering
  	\caption{  Convergence orders and errors  of the spectral Petrov-Galerkin  method \eqref{PG-scheme} for Example \ref{spectral-ex-1} with  $f= \sin x.$  }\label{table:spectral-PG-exam-1}
  	\begin{tabular}{cccccccccc}
  		\hline\hline
  		&	& \multicolumn{2}{c}{ $\alpha=1.2$}&   \multicolumn{2}{c}{ $\alpha=1.4$}  &\multicolumn{2}{c}{ $\alpha=1.6$}&\multicolumn{2}{c}{ $\alpha=1.8$} \\
  		\cline{3-4} \cline{5-6} \cline{7-8}\cline{9-10}
  		&	$N$ &  $E_1(N) $ & rate &  $E_1(N) $ & rate  &$E_1(N) $ & rate & $E_1(N) $ & rate  \\
  		\hline
  		$\theta=0.5$ & 16       & 1.29e-04     &              & 1.11e-05     &              & 1.43e-06     &          & 1.78e-07     &          \\ 
  		& 32       & 1.31e-05     &     3.31     & 8.81e-07     &     3.65     & 8.82e-08     &     4.01     & 8.54e-09     &     4.38    \\ 
  		& 64       & 1.29e-06     &     3.34     & 6.71e-08     &     3.71     & 5.17e-09     &     4.09     & 3.85e-10     &     4.47    \\ 
  		& 128      & 1.25e-07     &     3.37     & 4.97e-09     &     3.75     & 2.92e-10     &     4.14     & 1.66e-11     &     4.53    \\ 
  		\hline
  		Order  &(Averaged)&& 3.34 && 3.71 && 4.08 && 4.46\\
  		\hline
  		$2\alpha+1$  & (Theorem \ref{thm:spectral-error-estimate-2}) && 3.40 && 3.80 && 4.20 && 4.60\\ 	\hline \hline
  		$\theta=0.7$ &  16       & 3.84e-05     &          & 7.39e-06     &          & 1.24e-06     &          & 1.75e-07     &          \\ 
  		&  32       & 3.90e-06     &     3.30     & 5.82e-07     &     3.67     & 7.68e-08     &     4.02     & 8.46e-09     &     4.37    \\ 
  		&  64       & 3.80e-07     &     3.36     & 4.40e-08     &     3.73     & 4.50e-09     &     4.09     & 3.83e-10     &     4.47    \\ 
  		&  128       & 3.60e-08     &     3.40     & 3.23e-09     &     3.77     & 2.54e-10     &     4.15     & 1.66e-11     &     4.53    \\ 
  		\hline
  		Order  &(Averaged) && 3.35 && 3.72 && 4.07 && 4.46\\
  		\hline
  		$2\alpha+1$  & (Theorem \ref{thm:spectral-error-estimate-2})  && 3.40 && 3.80 && 4.20 && 4.60\\ 	\hline
  		\hline
  		$\theta=1$	 & 	16       & 3.87e-06     &          & 1.99e-06     &          & 7.04e-07     &          & 1.63e-07     &          \\ 
  		& 	32       & 3.94e-07     &     3.29     & 1.56e-07     &     3.67     & 4.42e-08     &     4.00     & 8.14e-09     &     4.32    \\ 
  		& 	64       & 3.80e-08     &     3.38     & 1.17e-08     &     3.74     & 2.61e-09     &     4.08     & 3.75e-10     &     4.44    \\ 
  		& 	128       & 3.55e-09     &     3.42     & 8.56e-10     &     3.78     & 1.49e-10     &     4.13     & 1.65e-11     &     4.51    \\ 
  		\hline
  		Order  & (Averaged)&& 3.36 && 3.73 && 4.07 && 4.42\\ \hline
  		$2\alpha+1$  & (Theorem \ref{thm:spectral-error-estimate-2}) && 3.40 && 3.80 && 4.20 && 4.60\\
  		\hline
  		\hline
  	\end{tabular}
  \end{table}

  \begin{table}[!htb]
  	\centering
  	\caption{  Convergence orders and errors  of the spectral Galerkin  method \eqref{weak-form-scheme_3} for Example \ref{spectral-ex-1}  with  $f= \sin x.$   }\label{table:spectral-G-exam-1}
  	\begin{tabular}{cccccccccc}
  		\hline\hline
  		&	& \multicolumn{2}{c}{ $\alpha=1.2$}&   \multicolumn{2}{c}{ $\alpha=1.4$}  &\multicolumn{2}{c}{ $\alpha=1.6$}&\multicolumn{2}{c}{ $\alpha=1.8$} \\
  		\cline{3-4} \cline{5-6} \cline{7-8}\cline{9-10}
  		&	$N$ &  $E_1(N) $ & rate &  $E_1(N) $ & rate  &$E_1(N) $ & rate & $E_1(N) $ & rate  \\
  		\hline
  		$\theta=0.7$ & 16       & 3.25e-04     &          & 2.22e-05     &          & 3.09e-06     &          & 5.04e-07     &          \\ 
  		&32       & 5.41e-05     &     2.59     & 1.83e-06     &     3.60     & 1.86e-07     &     4.05     & 2.42e-08     &     4.38    \\ 
  		& 64       & 8.66e-06     &     2.64     & 1.41e-07     &     3.70     & 1.06e-08     &     4.13     & 1.08e-09     &     4.48    \\ 
  		&128       & 1.35e-06     &     2.69     & 1.04e-08     &     3.76     & 5.84e-10     &     4.18     & 4.67e-11     &     4.54    \\ 
  		\hline
  		Order  &(Averaged)&& 2.64 && 3.69 && 4.12 && 4.47\\ 	\hline 	
  		$2\alpha+1$  &  && 3.40 && 3.80 && 4.20 && 4.60\\	\hline
  		$\alpha+1$  & (Theorem \ref{thm:spectral-error-estimate-1}) && 2.20 && 2.40 && 2.60 && 2.80\\\hline
  		\hline 	
  		$\theta=1$	 & 16       & 3.45e-04     &          & 4.77e-05     &          & 4.77e-06     &          & 5.21e-07     &          \\ 
  		& 32       & 9.45e-05     &     1.87     & 6.76e-06     &     2.82     & 3.51e-07     &     3.76     & 2.47e-08     &     4.40    \\ 
  		& 64       & 2.48e-05     &     1.93     & 9.07e-07     &     2.90     & 2.40e-08     &     3.87     & 1.09e-09     &     4.50    \\ 
  		& 128       & 6.23e-06     &     1.99     & 1.17e-07     &     2.95     & 1.57e-09     &     3.93     & 4.63e-11     &     4.56    \\ 
  		\hline
  		Order  &(Averaged)&& 1.93 && 2.89 && 3.85 && 4.47\\
  		\hline
  		$2\alpha+1$  &  && 3.40 && 3.80 && 4.20 && 4.60\\	\hline
  		$\alpha+1$  & (Theorem \ref{thm:spectral-error-estimate-1}) && 2.20 && 2.40 && 2.60 && 2.80\\
  		\hline
  		\hline
  	\end{tabular}
  \end{table}

 \begin{example}\label{spectral-ex-2}
 	Consider	$f= \abs{\sin x}.$ The function $f$ has a weak singularity  at $x=0$ and    $f\in  B^{1.5-\epsilon}_{\omega^{\sigma^{\ast}-1,\sigma-1}}$ for any   $\epsilon >0.$ By Theorem \ref{thm:fode-regularity-two-side},   	$\omega^{-\sigma,-\sigma^{\ast}}u \in   B^{ \alpha+1.5-\epsilon }_{\omega^{ \sigma,\sigma^{\ast} }}$. 
 \end{example}
 
  According to Theorems  \ref{thm:spectral-error-estimate-2} and   \ref{thm:spectral-error-estimate-1}, 
  the convergence orders are expected to be $\alpha+1.5-\epsilon$ for the spectral Petrov-Galerkin   method \eqref{PG-scheme} and $1.5-\epsilon$ for the spectral Galerkin  method \eqref{weak-form-scheme_3}.

 From Table \ref{table:spectral-PG-exam-2}, we can observe that the convergence order is $\alpha+1.5-\epsilon$ for the spectral Petrov-Galerkin   method \eqref{PG-scheme}, which is in agreement with the theoretical prediction 
   when  the order $\alpha=1.2,\,1.4,\,1.6,\,1.8$.  
   
    In Table   \ref{table:spectral-G-exam-2}, we observe that the convergence orders for the spectral Galerkin method lie in between $\alpha+1.5-\epsilon$ and 
    $1.5 -\epsilon$.  
    For  
    $\theta=0.7$,    the observed convergence order is   $\alpha+1.5-\epsilon$ when  $\alpha=1.8.$  However, we observed that the convergence orders decrease with $\alpha$
    when  $\theta=0.7$ and $\theta=1$. 
      
In this example,  the spectral Petrov-Galerkin  method \eqref{PG-scheme} has  
 the convergence order $\alpha+1.5-\epsilon$, which  suggests the regularity index $\alpha+1.5-\epsilon$ for the solution and    verifies  Theorem \ref{thm:fode-regularity-two-side}.  The spectral Petrov-Galerkin  method  has higher accuracy  than the  spectral Galerkin   method  \eqref{weak-form-scheme_3}, especially for small fractional order $\alpha$ when $\theta \neq 0.5$.  Again better  error estimates for the spectral Galerkin method should be considered for further research.
 
 \begin{table}[!htb]
 	\centering
 	\caption{  Convergence orders  and errors of the  spectral Petrov-Galerkin method  \eqref{PG-scheme} for Example \ref{spectral-ex-2}  with  $f= |\sin x|.$     }\label{table:spectral-PG-exam-2} 
 	\begin{tabular}{cccccccccc}
 		\hline\hline
 		&	& \multicolumn{2}{c}{ $\alpha=1.2$}&   \multicolumn{2}{c}{ $\alpha=1.4$}  &\multicolumn{2}{c}{ $\alpha=1.6$}&\multicolumn{2}{c}{ $\alpha=1.8$} \\
 		\cline{3-4} \cline{5-6} \cline{7-8}\cline{9-10}
 		&	$N$ &  $E_1(N) $ & rate &  $E_1(N) $ & rate  &$E_1(N) $ & rate & $E_1(N) $ & rate  \\
 		
 		\hline
 		$\theta=0.5$ & 16       & 1.60e-03     &          & 4.88e-04     &          & 2.12e-04     &          & 1.11e-04     &          \\ 
 		& 32       & 2.70e-04     &     2.56     & 7.32e-05     &     2.74     & 2.85e-05     &     2.89     & 1.36e-05     &     3.03    \\ 
 		& 64       & 4.28e-05     &     2.66     & 1.01e-05     &     2.85     & 3.48e-06     &     3.03     & 1.49e-06     &     3.19    \\ 
 		& 128       & 6.61e-06     &     2.70     & 1.35e-06     &     2.91     & 4.05e-07     &     3.10     & 1.54e-07     &     3.27    \\ 
 		\hline
 		Order  &(Averaged) && 2.64 && 2.83 && 3.01 && 3.16\\
 		\hline 
 		$\alpha+1.5-\epsilon$  &(Theorem \ref{thm:spectral-error-estimate-2}) && 2.70 && 2.90 && 3.10 && 3.30\\
 		\hline\hline
 		$\theta=0.7$ & 16       & 1.43e-03     &          & 5.05e-04     &          & 2.19e-04     &          & 1.12e-04     &          \\ 
 		& 32       & 2.50e-04     &     2.51     & 7.73e-05     &     2.71     & 2.98e-05     &     2.87     & 1.38e-05     &     3.02    \\ 
 		& 64       & 4.00e-05     &     2.64     & 1.08e-05     &     2.84     & 3.69e-06     &     3.02     & 1.52e-06     &     3.18    \\ 
 		& 128       & 6.11e-06     &     2.71     & 1.44e-06     &     2.90     & 4.33e-07     &     3.09     & 1.58e-07     &     3.26    \\ 
 		\hline
 		Order  &(Averaged)&& 2.62 && 2.82 && 2.99 && 4.15\\
 		\hline 
 		$\alpha+1.5-\epsilon$  &(Theorem \ref{thm:spectral-error-estimate-2})&& 2.70 && 2.90 && 3.10 && 3.30\\
 		\hline \hline
 		$\theta=1$& 16       & 9.36e-04     &          & 4.69e-04     &          & 2.32e-04     &          & 1.17e-04     &          \\ 
 		& 32       & 1.69e-04     &     2.47     & 7.52e-05     &     2.64     & 3.30e-05     &     2.81     & 1.47e-05     &     2.99    \\ 
 		& 64       & 2.81e-05     &     2.59     & 1.10e-05     &     2.77     & 4.25e-06     &     2.95     & 1.66e-06     &     3.15    \\ 
 		& 128       & 4.49e-06     &     2.65     & 1.54e-06     &     2.84     & 5.20e-07     &     3.03     & 1.77e-07     &     3.23    \\ 
 		\hline
 		Order  &(Averaged) && 2.57 && 2.75 && 2.93 && 3.12\\
 		\hline 
 		$\alpha+1.5-\epsilon$  &(Theorem \ref{thm:spectral-error-estimate-2})&& 2.70 && 2.90 && 3.10 && 3.30\\
 		\hline
 		\hline
 	\end{tabular}
 \end{table}

 \begin{table}[!htb]
 	\centering
 	\caption{  Convergence orders and errors  of the spectral Galerkin  method \eqref{weak-form-scheme_3} for Example \ref{spectral-ex-2} with  $f= |\sin x|.$   }\label{table:spectral-G-exam-2}
 	\begin{tabular}{cccccccccc}
 		\hline\hline
 		&	& \multicolumn{2}{c}{ $\alpha=1.2$}&   \multicolumn{2}{c}{ $\alpha=1.4$}  &\multicolumn{2}{c}{ $\alpha=1.6$}&\multicolumn{2}{c}{ $\alpha=1.8$} \\
 		\cline{3-4} \cline{5-6} \cline{7-8}\cline{9-10}
 		&	$N$ &  $E_1(N) $ & rate &  $E_1(N) $ & rate  &$E_1(N) $ & rate & $E_1(N) $ & rate  \\
 		\hline
 		$\theta=0.7$ & 16       & 1.73e-03     &          & 7.77e-04     &          & 3.53e-04     &          & 1.96e-04     &          \\ 
 		&32       & 4.47e-04     &     1.96     & 1.40e-04     &     2.48     & 5.18e-05     &     2.77     & 2.47e-05     &     2.99    \\ 
 		& 64       & 1.06e-04     &     2.08     & 2.29e-05     &     2.61     & 6.90e-06     &     2.91     & 2.80e-06     &     3.14    \\ 
 		& 128       & 2.36e-05     &     2.17     & 3.56e-06     &     2.68     & 8.73e-07     &     2.98     & 3.00e-07     &     3.22    \\ 
 		\hline
 		Order  &(Averaged)&& 2.07 && 2.57 && 2.89 && 3.12\\  \hline
 		
 		$\alpha+1.5-\epsilon$  &&& 2.70 && 2.90 && 3.10 && 3.30\\ 	\hline 
 		$1.5-\epsilon$  & (Theorem \ref{thm:spectral-error-estimate-1})&& 1.50 && 1.50 && 1.50 && 1.50\\
 		\hline \hline
 		$\theta=1$	&16       & 7.04e-04     &          & 7.44e-04     &          & 4.28e-04     &          & 2.18e-04     &          \\ 
 		& 32       & 1.82e-04     &     1.95     & 1.79e-04     &     2.05     & 7.78e-05     &     2.46     & 3.03e-05     &     2.84    \\ 
 		&64       & 5.12e-05     &     1.83     & 4.01e-05     &     2.16     & 1.30e-05     &     2.58     & 3.84e-06     &     2.98    \\ 
 		& 128       & 1.44e-05     &     1.83     & 8.46e-06     &     2.24     & 2.07e-06     &     2.65     & 4.62e-07     &     3.05    \\ 
 		\hline
 		Order  &(Averaged)&& 1.87 && 2.15 && 2.56 && 2.96\\ 
 		\hline 
 		$\alpha+1.5-\epsilon$  &&& 2.70 && 2.90 && 3.10 && 3.30\\  \hline
 		$1.5-\epsilon$  & (Theorem \ref{thm:spectral-error-estimate-1})&& 1.50 && 1.50 && 1.50 && 1.50\\
 		\hline
 		\hline
 	\end{tabular}
 \end{table}

 \begin{example} \label{spectral-ex-3}
 	Consider	 $f= (1-x^2)^\beta   \sin x$. Here $f \in B^{\sigma\wedge \sigma^{\ast}+2\beta}_{\omega^{\sigma^{\ast}-1,\sigma-1}}\cap \in B^{\sigma\wedge \sigma^{\ast}+2\beta+1}_{\omega^{\sigma^{\ast},\sigma}}$. From   Theorem \ref{thm:fode-regularity-two-side},   	$\omega^{-\sigma,-\sigma^{\ast}}u \in   B^{ \alpha+(\sigma\wedge {\sigma^{\ast}}+2\beta )\wedge (\alpha+1)}_{\omega^{ \sigma,\sigma^{\ast} }}$ and by Theorem \ref{thm:fode-regularity-flap-nonsymm}, $\omega^{-\sigma,-\sigma^{\ast}}u \in   B^{ \alpha+(\sigma\wedge {\sigma^{\ast}}+2\beta+1)\wedge \alpha}_{\omega^{ \sigma,\sigma^{\ast} }}.$   
 \end{example}
 In this example,  we measure the error using $E_2(N)$ instead
 of $E_1(N).$  We test the different $\beta$'s  in Tables
 \ref{table:spectral-PG-exam-3-case-1}-\ref{table:spectral-G-exam-3-case-1}
 ($\beta =0.5$)   and in Tables
 \ref{table:spectral-PG-exam-3-case-3}-\ref{table:spectral-G-exam-3-case-3}
 ($\beta = -0.4$). 
 
We first test the case $\beta=0.5$ where  the forcing term $f$ has weak singularity and vanishes at both end points $\pm 1$.
By Theorems \ref{thm:spectral-error-estimate-2} and  \ref{thm:spectral-error-estimate-1} the theoretical  orders for the Petrov-Galerkin method and the Galerkin method  are  $\alpha+\sigma\wedge {\sigma^{\ast}}+1$ and $\sigma\wedge {\sigma^{\ast}}+1$.  However, 
we observe in  Tables
\ref{table:spectral-PG-exam-3-case-1}-\ref{table:spectral-G-exam-3-case-1}
that both methods can have  convergence orders as high as   $\alpha+\sigma\wedge {\sigma^{\ast}}+2$.

We then consider the  singular forcing term  $f=(1-x^2)^\beta \sin x$ where   $\beta=-0.4$. In this case, we apply Theorem \ref{thm:fode-regularity-flap-nonsymm} instead of Theorem \ref{thm:fode-regularity-two-side} in order to get higher regularity index. 
From   Theorem \ref{thm:fode-regularity-two-side},   	$\omega^{-\sigma,-\sigma^{\ast}}u \in   B^{ \alpha+(\sigma\wedge {\sigma^{\ast}}+2\beta )\wedge (\alpha+1)}_{\omega^{ \sigma,\sigma^{\ast} }}$ and by Theorem \ref{thm:fode-regularity-flap-nonsymm}, $\omega^{-\sigma,-\sigma^{\ast}}u \in   B^{ \alpha+(\sigma\wedge {\sigma^{\ast}}+2\beta+1)\wedge \alpha}_{\omega^{ \sigma,\sigma^{\ast} }}.$  For $\beta =-0.4$, $\omega^{-\sigma,-\sigma^{\ast}}u \in   B^{ \alpha+(\sigma\wedge {\sigma^{\ast}}+0.2)\wedge \alpha}_{\omega^{ \sigma,\sigma^{\ast} }}$.
According to Corollary \ref{cor:spectral-error-estimate-2}  the 
convergence  orders for the spectral Petrov-Galerkin method   are  $\alpha+\sigma\wedge {\sigma^{\ast}}+0.2$, which is  demonstrated in  Table \ref{table:spectral-PG-exam-3-case-3}.
From Corollary  \ref{cor:spectral-error-estimate-1}, the 
convergence  orders for the spectral Galerkin method  are expected to be $\sigma\wedge {\sigma^{\ast}}+0.2$.   However, the observed orders are  $\alpha+\sigma\wedge {\sigma^{\ast}}+0.2$ in Table \ref{table:spectral-G-exam-3-case-3}.

 In this example, the convergence orders for the two methods are almost the same 
 when  the forcing term $f(x)$ has both weak boundary
 singularity ($\beta=0.5$) or stronger boundary singularity ($\beta=-0.4$), which  suggest the error estimate for  the  Galerkin method can be improved in the non-symmetrical case $\theta\neq 0.5$.    For the Petrov-Galerkin method, however, the convergence orders are higher than the theoretical predictions in the case of $\beta=0.5$.

 \begin{table}[!htb]
 	\centering
 	\caption{  Convergence orders and errors  of the spectral  Petrov-Galerkin   method \eqref{PG-scheme} for Example \ref{spectral-ex-3} with $f= (1-x^2)^{0.5}\sin x.$    }\label{table:spectral-PG-exam-3-case-1}
 \begin{tabular}{cccccccccc}
 	\hline\hline
 	&	& \multicolumn{2}{c}{ $\alpha=1.2$}&   \multicolumn{2}{c}{ $\alpha=1.4$}  &\multicolumn{2}{c}{ $\alpha=1.6$}&\multicolumn{2}{c}{ $\alpha=1.8$} \\
 	\cline{3-4} \cline{5-6} \cline{7-8}\cline{9-10}
 	&	$N$ &  $E_2(N) $ & rate &  $E_2(N) $ & rate  &$E_2(N) $ & rate & $E_2(N) $ & rate  \\
 \hline
 $\theta=0.5$  &16        & 4.00e-05     &          & 1.13e-05     &          & 3.95e-06     &          & 1.57e-06     &          \\ 
 & 32         & 3.32e-06     &     3.59     & 7.73e-07     &     3.87     & 2.23e-07     &     4.15     & 7.32e-08     &     4.42    \\ 
 & 64     & 2.61e-07     &     3.68     & 4.93e-08     &     3.99     & 1.16e-08     &     4.29     & 3.12e-09     &     4.58    \\ 
 & 128        & 1.99e-08     &     3.71     & 3.03e-09     &     4.03     & 5.76e-10     &     4.33     & 1.26e-10     &     4.63    \\ 
 \hline
 Order  &(Averaged)&& 3.66 && 3.96 && 4.26 && 4.54\\
 \hline
 $\alpha+\sigma\wedge \sigma^{\ast}+2$  &&& 3.80 && 4.10 && 4.40 && 4.70\\
 \hline
 $\alpha+\sigma\wedge \sigma^{\ast}+1$  &(Theorem \ref{thm:spectral-error-estimate-2})&& 2.80 && 3.10 && 3.40 && 3.70\\
 	\hline \hline
 	$\theta=0.7$&16        & 4.49e-05     &          & 1.16e-05     &          & 3.96e-06     &          & 1.57e-06     &          \\ 
 	& 32      & 4.46e-06     &     3.33     & 8.57e-07     &     3.76     & 2.29e-07     &     4.11     & 7.36e-08     &     4.41    \\ 
 	&64       & 4.20e-07     &     3.42     & 5.98e-08     &     3.86     & 1.23e-08     &     4.24     & 3.16e-09     &     4.57    \\ 
 	& 128       & 3.83e-08     &     3.45     & 4.03e-09     &     3.89     & 6.38e-10     &     4.27     & 1.29e-10     &     4.61    \\ 
 	\hline
 	Order  &(Averaged)&& 3.40 && 3.82 && 4.21 && 4.53\\
 		\hline
 		$\alpha+\sigma\wedge \sigma^{\ast}+2$  &&& 3.52 && 3.94 && 4.31 && 4.66\\
 	\hline
 	$\alpha+\sigma\wedge \sigma^{\ast}+1$  &(Theorem \ref{thm:spectral-error-estimate-2})&& 2.52 && 2.94 && 3.31 && 3.66\\
 	\hline \hline
 	$\theta=1$  &16       & 3.04e-05     &          & 1.05e-05     &          & 3.88e-06     &          & 1.56e-06     &          \\ 
 	& 32       & 3.28e-06     &     3.22     & 8.65e-07     &     3.60     & 2.43e-07     &     3.99     & 7.53e-08     &     4.37    \\ 
 	& 64        & 3.34e-07     &     3.31     & 6.73e-08     &     3.70     & 1.43e-08     &     4.11     & 3.34e-09     &     4.52    \\ 
 	& 128         & 3.29e-08     &     3.34     & 5.05e-09     &     3.74     & 8.10e-10     &     4.14     & 1.42e-10     &     4.56    \\ 
 	\hline
 	Order  &(Averaged)&& 3.29 && 3.68 && 4.08 && 4.48\\
 		\hline 
 		$\alpha+\sigma\wedge \sigma^{\ast}+2$  &&& 3.40 && 3.80 && 4.20 && 4.60\\
 	\hline 
 	$\alpha+\sigma\wedge \sigma^{\ast}+1$  &(Theorem \ref{thm:spectral-error-estimate-2})&& 2.40 && 2.80 && 3.20 && 3.60\\
 	\hline
 	\hline
 \end{tabular}
\end{table}

  \begin{table}[!htb]
  	\centering
  	\caption{  Convergence orders and errors  of the spectral Galerkin  method \eqref{weak-form-scheme_3} for Example \ref{spectral-ex-3}   with  $f= (1-x^2)^{0.5}\sin x.$   }\label{table:spectral-G-exam-3-case-1}
  	\begin{tabular}{cccccccccc}
  		\hline\hline
  		&	& \multicolumn{2}{c}{ $\alpha=1.2$}&   \multicolumn{2}{c}{ $\alpha=1.4$}  &\multicolumn{2}{c}{ $\alpha=1.6$}&\multicolumn{2}{c}{ $\alpha=1.8$} \\
  		\cline{3-4} \cline{5-6} \cline{7-8}\cline{9-10}
  		&	$N$ &  $E_2(N) $ & rate &  $E_2(N) $ & rate  &$E_2(N) $ & rate & $E_2(N) $ & rate  \\
  		\hline
  		$\theta=0.7$& 16         & 6.54e-05     &          & 1.70e-05     &          & 6.14e-06     &          & 2.71e-06     &          \\ 
  		& 32       & 6.08e-06     &     3.43     & 1.25e-06     &     3.77     & 3.55e-07     &     4.11     & 1.27e-07     &     4.41    \\ 
  		& 64       & 5.42e-07     &     3.49     & 8.52e-08     &     3.87     & 1.90e-08     &     4.23     & 5.45e-09     &     4.54    \\ 
  		& 128        & 4.74e-08     &     3.52     & 5.64e-09     &     3.92     & 9.71e-10     &     4.29     & 2.23e-10     &     4.61    \\ 
  		\hline
  		Order  &(Averaged)&& 3.48 && 3.85 && 4.21 && 4.52\\
  		\hline
  		$\alpha+\sigma \wedge \sigma^{\ast}+2$  & && 3.52 && 4.94 && 4.31 && 4.66\\ 
  		\hline
  		$\sigma \wedge \sigma^{\ast}+1$& (Theorem \ref{thm:spectral-error-estimate-1}) && 1.32 && 1.54 && 1.71 && 1.86\\
  		\hline \hline
  		$\theta=1$& 16         & 5.52e-05     &          & 1.86e-05     &          & 6.78e-06     &          & 2.77e-06     &          \\ 
  		& 32       & 5.10e-06     &     3.44     & 1.47e-06     &     3.66     & 4.18e-07     &     4.02     & 1.33e-07     &     4.38    \\  
  		&64       & 4.18e-07     &     3.61     & 1.09e-07     &     3.75     & 2.38e-08     &     4.13     & 5.83e-09     &     4.51    \\ 
  		&128       & 3.20e-08     &     3.70     & 7.84e-09     &     3.80     & 1.31e-09     &     4.19     & 2.44e-10     &     4.58    \\ 
  		\hline
  		Order  &(Averaged)&& 3.58 && 3.74 && 4.11 && 4.49\\  
  		\hline
  		$\alpha+\sigma \wedge \sigma^{\ast}+2$  & && 3.52 && 3.94 && 4.31 && 4.66\\  
  		\hline
  		$\sigma \wedge \sigma^{\ast}+1$& (Theorem \ref{thm:spectral-error-estimate-1}) && 1.20 && 1.40 && 1.60 && 1.80\\	
  		
  		\hline
  		\hline
  	\end{tabular}
  \end{table}

  \begin{table}[!htb]
  	\centering
  	\caption{  Convergence orders and errors of the  spectral Petrov-Galerkin method \eqref{PG-scheme} for Example \ref{spectral-ex-3}  with $f= (1-x^2)^{-0.4}\sin x.$   }\label{table:spectral-PG-exam-3-case-3}
  	\begin{tabular}{cccccccccc}
  		\hline\hline
  		&	& \multicolumn{2}{c}{ $\alpha=1.2$}&   \multicolumn{2}{c}{ $\alpha=1.4$}  &\multicolumn{2}{c}{ $\alpha=1.6$}&\multicolumn{2}{c}{ $\alpha=1.8$} \\
  		\cline{3-4} \cline{5-6} \cline{7-8}\cline{9-10}
  		&	$N$ &  $E_2(N) $ & rate &  $E_2(N) $ & rate  &$E_2(N) $ & rate & $E_2(N) $ & rate  \\
  		\hline
  		$\theta=0.5$ &16           & 4.83e-03     &          & 1.03e-03     &          & 3.11e-04     &          & 1.12e-04     &          \\ 
  		&32         & 1.27e-03     &     1.92     & 2.26e-04     &     2.18     & 5.67e-05     &     2.45     & 1.69e-05     &     2.72    \\ 
  		&64      & 3.28e-04     &     1.96     & 4.79e-05     &     2.24     & 9.86e-06     &     2.53     & 2.42e-06     &     2.81    \\ 
  		&128        & 8.31e-05     &     1.98     & 9.92e-06     &     2.27     & 1.67e-06     &     2.56     & 3.34e-07     &     2.85    \\ 
  		\hline
  		Order & (Averaged) && 1.95 && 2.23 && 2.51 && 2.79\\  \hline 
  		$\alpha+\sigma \wedge \sigma^{\ast} +0.2$  & (Theorem \ref{thm:spectral-error-estimate-2})
  		&	& 2.00 && 2.30 && 2.60 && 2.90 \\ \hline \hline  
  		
  		$\theta=0.7$ &	16         & 5.75e-03     &          & 1.13e-03     &          & 3.20e-04     &          & 1.12e-04     &          \\ 
  		&	32       & 1.86e-03     &     1.63     & 2.73e-04     &     2.05     & 6.05e-05     &     2.40     & 1.72e-05     &     2.71    \\ 
  		&	64      & 5.85e-04     &     1.67     & 6.37e-05     &     2.10     & 1.09e-05     &     2.47     & 2.47e-06     &     2.80    \\ 
  		&	128        & 1.80e-04     &     1.70     & 1.47e-05     &     2.12     & 1.94e-06     &     2.50     & 3.45e-07     &     2.84    \\ 
  		\hline
  		Order  &   (Averaged)                      && 1.67 && 2.09 && 2.46 && 2.78\\
  		\hline
  		$\alpha+\sigma \wedge \sigma^{\ast}+0.2 $  &   			(Theorem \ref{thm:spectral-error-estimate-2})
  		&& 1.72 && 2.14 && 2.51 && 2.86\\
  		\hline \hline
  		
  		$\theta=1$ &	16         & 4.42e-03     &          & 1.16e-03     &          & 3.43e-04     &          & 1.15e-04     &          \\ 
  		&	32       & 1.56e-03     &     1.50     & 3.15e-04     &     1.88     & 7.09e-05     &     2.27     & 1.82e-05     &     2.67    \\ 
  		& 64 & 5.33e-04     &     1.55     & 8.23e-05     &     1.94     & 1.40e-05     &     2.33     & 2.71e-06     &     2.74    \\ 
  		&	128         & 1.78e-04     &     1.58     & 2.10e-05     &     1.97     & 2.72e-06     &     2.37     & 3.95e-07     &     2.78    \\ 
  		\hline
  		Order  &   (Averaged)                   && 1.55 && 1.93 && 2.32 && 2.73\\
  		\hline
  		$\alpha+\sigma \wedge \sigma^{\ast}+0.2 $    & (Theorem \ref{thm:spectral-error-estimate-2}) && 1.60 && 2.00 && 2.40 && 2.80\\
  		\hline
  		\hline
  	\end{tabular}
  \end{table}

  \begin{table}[!htb]
  	\centering
  	\caption{  Convergence orders and errors of the spectral Galerkin  method \eqref{weak-form-scheme_3} for Example \ref{spectral-ex-3}    with  $f= (1-x^2)^{-0.4}\sin x.$  }\label{table:spectral-G-exam-3-case-3}
  	\begin{tabular}{cccccccccc}
  		\hline\hline
  		&	& \multicolumn{2}{c}{ $\alpha=1.2$}&   \multicolumn{2}{c}{ $\alpha=1.4$}  &\multicolumn{2}{c}{ $\alpha=1.6$}&\multicolumn{2}{c}{ $\alpha=1.8$} \\
  		\cline{3-4} \cline{5-6} \cline{7-8}\cline{9-10}
  		&	$N$ &  $E_2(N) $ & rate &  $E_2(N) $ & rate  &$E_2(N) $ & rate & $E_2(N) $ & rate  \\

  		\hline
  		$\theta=0.7$& 16      & 7.90e-03     &          & 1.58e-03     &          & 4.90e-04     &         & 1.94e-04     &         \\ 
  	
  		& 32           & 2.41e-03     &     1.71     & 3.68e-04     &     2.10     & 9.11e-05     &     2.43     & 2.95e-05     &     2.72    \\ 
  	 
  		& 	64        & 7.15e-04     &     1.75     & 8.30e-05     &     2.15     & 1.62e-05     &     2.49     & 4.23e-06     &     2.80    \\ 
  		&128       & 2.08e-04     &     1.79     & 1.85e-05     &     2.17     & 2.83e-06     &     2.52     & 5.89e-07     &     2.85    \\ 
  			\hline
  			Order  &(Averaged)&& 1.75 && 2.14 && 2.48 && 2.79\\ 
  			\hline
  		 $\alpha+\sigma \wedge \sigma^{\ast}+0.2$  &&& 1.72 && 2.14 && 2.51 && 2.86\\
  		\hline
  		 $\sigma \wedge \sigma^{\ast}+0.2$  &(Theorem \ref{thm:spectral-error-estimate-1})&& 0.52 && 0.74 && 0.91 && 1.06\\
  		 \hline \hline 
  		$\theta=1$ &16       & 7.05e-03     &          & 1.81e-03     &         & 5.47e-04     &         & 1.99e-04     &         \\ 
  		&32       & 2.37e-03     &     1.57     & 4.60e-04     &     1.97     & 1.08e-04     &     2.34     & 3.09e-05     &     2.69    \\ 
  		&64       & 7.69e-04     &     1.63     & 1.13e-04     &     2.02     & 2.04e-05     &     2.40     & 4.53e-06     &     2.77    \\ 
  		&128       & 2.40e-04     &     1.68     & 2.72e-05     &     2.06     & 3.80e-06     &     2.43     & 6.48e-07     &     2.81    \\ 
  		
  			\hline
  			Order  &(Averaged)&& 1.63 && 2.02 && 2.39 && 2.76\\
  				\hline
  				$\alpha+\sigma \wedge \sigma^{\ast}+0.2$  & && 1.60 && 2.00 && 2.40 && 2.80\\
  				\hline
  		$\sigma \wedge \sigma^{\ast}+0.2$   &(Theorem \ref{thm:spectral-error-estimate-1})&& 0.40 && 0.60 && 0.80 && 1.00\\	
  		
  		\hline
  		\hline
  	\end{tabular}
  \end{table}

 \section{Conclusion}
 In this paper, we discuss the regularity of the two-sided fractional diffusion equations with Riemann-Liouville operators under the homogeneous Dirichlet  boundary conditions.  
 Writing  $u =\omega^{\sigma,\sigma^{\ast}}\tilde{u}$, we find that the regularity index of $\tilde{u}$  is shown to be $2\alpha+1$.
 We also validate our finding by considering two numerical methods:  spectral Galerkin and  Petrov-Galerkin  methods.
 With the regularity index, we  showed the optimal error estimate
 for both methods when $\theta=0.5$. 
For  $\theta\neq 0.5$,   we obtained the optimal error estimate
  when the reaction coefficient $\mu$ is small.

  The error estimate for the Galerkin method are not optimal
  and  the estimate for the  Petrov-Galerkin method requires
  some additional condition on $\mu$. Further improvement in
  the error estimates are needed.   The analysis in this paper can be extended     to  FDEs with different low-order terms, such as FDEs with an  advection term.

 \appendix
 \section{Some useful relations of Jacobi polynomials}
 \label{sec:appendix-jacobi-polynomials}

 The following   relations  hold  for Jacobi polynomials $P_{n}^{\alpha,\beta}(x)$, see e.g.  \cite[Chapter 2]{Askey-B75}, 
 \begin{eqnarray}
 \partial_x P_{n}^{\alpha,\beta}(x)& =& \frac{n+\alpha+\beta+1}{2} P_{n-1}^{\alpha+1,\beta+1}(x), \quad \alpha,\beta>-1.\label{eq:jacobi-derivatives}
 \end{eqnarray}
 By \eqref{eq:jacobi-derivatives}, we have 
 \begin{equation}\label{eq:relation-different-index}
 \partial_x^{l} P_{n}^{\alpha,\beta}(x) = d_{n,l}^{\alpha,\beta} P_{n-l}^{\alpha+l,\beta+l}(x), \quad \alpha,\beta>-1, n\geq l,\quad d_{n,l}^{\alpha,\beta}=\frac{\Gamma{(n+\alpha+\beta+l+1)}}{2^l\Gamma{(n+\alpha+\beta+1)}} .
 \end{equation}

 \begin{theorem}\label{with_its_derivative_relationship} 
 	The Jacobi polynomials $P_{n}^{\alpha,\beta}(x)$  satisfy
 	\begin{equation} P_{n}^{\alpha,\beta} =\widehat{A}_n^{\alpha,\beta}\partial_xP_{n-1}^{\alpha,\beta}+\widehat{B}_n^{\alpha,\beta}\partial_xP_{n}^{\alpha,\beta}
 	+\widehat{C}_{n}^{\alpha,\beta}\partial_xP_{n+1}^{\alpha,\beta},\,
      n\geq 0,
 	\end{equation}
    where $P_{-1}^{\alpha,\beta}\equiv 0$ and
 	\begin{eqnarray*}
 		&&\widehat{A}_n^{\alpha,\beta}=\frac{-2(n+\alpha)(n+\beta)}{(n+\alpha+\beta)(2n+\alpha+\beta)(2n+\alpha+\beta+1)},\\
 		&&\widehat{B}_n^{\alpha,\beta}=\frac{2(\alpha-\beta)}{(2n+\alpha+\beta)(2n+\alpha+\beta+2)},\qquad \widehat{C}_n^{\alpha,\beta}=\frac{2(n+\alpha+\beta+1)}{(2n+\alpha+\beta+1)(2n+\alpha+\beta+2)}.
 	\end{eqnarray*}
 \end{theorem}

The relation \eqref{eq:jacobi-derivatives} and Theorem \ref{with_its_derivative_relationship} lead to the following result.
 \begin{corollary}\label{cor:different-index}
 	The Jacobi polynomials $P_{n}^{\alpha,\beta}(x)$  satisfy
 	\begin{eqnarray}
 	P_n^{\alpha,\beta}=A_{n}^{\alpha,\beta}P_{n-2}^{\alpha+1,\beta+1}+B_{n}^{\alpha,\beta}P_{n-1}^{\alpha+1,\beta+1}+
      C_{n}^{\alpha,\beta}P_{n}^{\alpha+1,\beta+1},\, n\geq0,
 	\end{eqnarray}	
 	where we let    $A_0^{\alpha,\beta}=A_1^{\alpha,\beta}=B_0^{\alpha,\beta}=0$  and 
 	$P_{-2}^{\alpha+1,\beta+1}=P_{-1}^{\alpha+1,\beta+1}=0$ and  
 	\begin{eqnarray*}
 		&&A_n^{\alpha,\beta}=-\frac{(n+\alpha)(n+\beta)}{(2n+\alpha+\beta)(2n+\alpha+\beta+1)},  \quad 
 		 B_n^{\alpha,\beta}=\frac{(\alpha-\beta)(n+\alpha+\beta+1)}{(2n+\alpha+\beta)(2n+\alpha+\beta+2)},\\
 		&&C_n^{\alpha,\beta}=\frac{(n+\alpha+\beta+1)(n+\alpha+\beta+2)}{(2n+\alpha+\beta+1)(2n+\alpha+\beta+2)}.
 	\end{eqnarray*}	
 \end{corollary}

 \begin{lemma}\label{lem:integral-estimate}
 	For any $k\geq n\geq 0$, it holds that  $\abs{X_k^n}\leq C$ where  
 	\begin{eqnarray}
 	X_k^n:=\frac{(P_k^{\alpha+1,\beta+1},P_n^{\alpha,\beta})_{\omega^{\alpha,\beta}}}{h_n^{\alpha,\beta}},\quad  \alpha>-1, \quad  \beta >-1.
 	\end{eqnarray}
 	 and  $ h_n^{\alpha,\beta}$ is defined in \eqref{eq:orthogonality-of-jacobi-poly-normalized-constant}.   
 \end{lemma}
 \begin{proof}
 	By  \eqref{eq:spectral-different-index}, we get 
 	\begin{eqnarray}
 	\delta_{nk}=A_{k}^{\alpha,\beta}X_{k-2}^n+B_{k}^{\alpha,\beta}X_{k-1}^n+C_{k}^{\alpha,\beta}X_k^n.
 	\end{eqnarray}
 	Thus we have 
 	$$X_n^n=\frac{1}{C_n^{\alpha,\beta}},\quad X_{n+1}^n=-\frac{B_{n+1}^{\alpha,\beta}}{C_{n+1}^{\alpha,\beta}}X_n^n,\quad X_{k+2}^n=p_kX_{k+1}^n+q_kX_{k}^n,\quad k\geq n.$$
 	where  $p_k=-\frac{B_{k+2}^{\alpha,\beta}}{C_{k+2}^{\alpha,\beta}}$ and $q_k=-\frac{A_{k+2}^{\alpha,\beta}}{C_{k+2}^{\alpha,\beta}}.$
 	Denote $Y_k^n=(X_{k+1}^n,X_k^n)^{\top}$ and  $A_k=(p_k,q_k;1,0).$   Then we have $Y_{k+1}^n=A_kY_k^n.$ It follows that 
 	$$\|Y^n_{k+1}\|_{\infty}=\|A_kY_k^n\|_{\infty} \leq \|A_k\|_{\infty} \|Y^n_{k}\|_{\infty} = \max \{|1,|p_k|+q_k \} \|Y_k^n\|_{\infty}, $$
 	where $\|Y_k^n\|_{\infty}=\max\{|X_k^n|,|X_{k+1}^n| \}.$
 	Recalling  $A_k^{\alpha,\beta}$,  $B_k^{\alpha,\beta}$ and $C_k^{\alpha,\beta}$  in Corollary  \ref{cor:different-index}, we have  for $k\geq 2$
 	\begin{eqnarray*}
 		&&|p_{k-2}|=\frac{|B_k^{\alpha,\beta}|}{C_k^{\alpha,\beta}}= \frac{|\alpha-\beta|(2k+\alpha+\beta+1)}{(2k+\alpha+\beta)(k+\alpha+\beta+2)}=\frac{|\alpha-\beta|}{k}+\mathcal{O}(\frac{1}{k^2}),\quad \text{ and }\\
 		&& q_{k-2}=-\frac{A_k^{\alpha,\beta}}{C_k^{\alpha,\beta}}=\frac{(k+\alpha)(k+\beta)(2k+\alpha+\beta+2)}{(k+\alpha+\beta+1)(k+\alpha+\beta+2)(2k+\alpha+\beta)}=1-\frac{\alpha+\beta+2}{k}+\mathcal{O}(\frac{1}{k^2}).
 	\end{eqnarray*}
 	Thus we arrive at
 	$$|p_k|+q_k= 1-\frac{\alpha+\beta+2-|\alpha-\beta|}{k}+\mathcal{O}(\frac{1}{k^2})=1-\frac{2\min(\alpha,\beta)+2}{k}+\mathcal{O}(\frac{1}{k^2}).$$
 	Since $\alpha>-1$ and $\beta>-1$,     $|p_k|+q_k\leq1.$  
 	Thus for any $k\geq n\geq 0$, we have $\|Y_{k+1}\|_{\infty}\leq  \|Y_k\|_{\infty},$ which leads to the desired result. 
 \end{proof}

The following asymptotic formula for a ratio of two gamma functions holds  
\begin{equation}\label{estimate-gamma}
\lim_{n\rightarrow \infty}\frac{\Gamma(n+\delta)}{n^{\delta-\gamma}\Gamma(n+\gamma)}=\lim_{n\rightarrow \infty}[1+\frac{(\delta-\gamma)(\delta+\gamma-1)}{2n}+\mathcal{O}(n^{-2})]=1.
\end{equation}

\section*{Acknowledgments}
Z. Hao would like to acknowledge the support by National Natural Science Foundation of
China (No. 11671083),  China Scholarship Council (No. 201506090065),  the Research and
 Innovation Project for College Graduates of Jiangsu Province (No. KYLX\_0081) and
 Natural Science Youth Foundation of Jiangsu Province (BK20160660).
G. Lin would like to acknowledge the support by the NSF Grant DMS-1555072.
Z. Zhang  would like to acknowledge the support by MURI and  a start-up fund from WPI.

\end{document}